\documentclass[a4paper,11pt]{article}

\usepackage{amsmath}
\usepackage{amssymb}
\usepackage{amsthm}
\usepackage{graphicx}
\usepackage{amscd}
\usepackage{epic, eepic}
\usepackage{url}
\usepackage[dvipsnames]{xcolor}
\usepackage[utf8]{inputenc} 
\usepackage{comment}
\usepackage{enumerate}   
\usepackage{graphicx}
\usepackage{epstopdf}
\usepackage{enumitem}
\usepackage[sort, numbers]{natbib}

\usepackage{todonotes}
\usepackage{soul,mathtools}

\usepackage[top=22mm, bottom=23mm, left=21mm, right=21mm]{geometry}
\usepackage[bookmarks=false,hidelinks]{hyperref}
\usepackage[capitalize]{cleveref}

\usepackage{tikz}
\usetikzlibrary{math,calc,intersections}
\usetikzlibrary{positioning,arrows,shapes,decorations.markings,decorations.pathreplacing, decorations.pathmorphing,matrix,patterns}
\tikzstyle{vertex}=[circle,draw=black,fill=black,inner sep=0,minimum size=5pt,text=white,font=\footnotesize]

\makeatletter
\renewenvironment{proof}[1][\proofname] {\par\pushQED{\qed}\normalfont\topsep6\p@\@plus6\p@\relax\trivlist\item[\hskip\labelsep\bfseries#1\@addpunct{.}]\ignorespaces}{\popQED\endtrivlist\@endpefalse}
\makeatother

\newtheorem{theorem}{\bf Theorem}[section]
\newtheorem{lemma}[theorem]{\bf Lemma}
\newtheorem{claim}[theorem]{\bf Claim}
\newtheorem{corollary}[theorem]{\bf Corollary}
\newtheorem{proposition}[theorem]{\bf Proposition}

\theoremstyle{definition}

\newtheorem{remark}[theorem]{\bf Remark}
\newtheorem{definition}[theorem]{\bf Definition}

\def\E{\mathbb{E}}
\def\Pb{\mathbb{P}}

\def\cP{\mathcal{P}}
\def\cQ{\mathcal{Q}}
\def\cS{\mathcal{S}}
\def\cT{\mathcal{T}}
\def\ex{\mathrm{ex}}

\DeclareMathOperator{\diam}{diam}

\newcommand{\ab}[1]{\lvert #1 \rvert}

\newlist{lemenum}{enumerate}{1}
\setlist[lemenum]{label=(\roman*), ref=\thelemma(\roman*)}
\crefalias{lemenumi}{lemma} 

\title{\vspace{-0.9cm}Canonical Ramsey numbers of sparse graphs}
\author{Lior Gishboliner\thanks{Department of Mathematics, University of Toronto. Email: {\tt lior.gishboliner@utoronto.ca}.} \and Aleksa Milojevi\'c \thanks{Department of Mathematics, ETH Z\"urich, Switzerland. Email: {\tt \{aleksa.milojevic, benjamin.sudakov\}@math.ethz.ch}. Research supported in part by SNSF grant 200021-228014.}\and Benny Sudakov \footnotemark[2]
\and Yuval Wigderson\thanks{Institute for Theoretical Studies, ETH Z\"urich, 8092 Z\"urich, Switzerland. 
   Supported by Dr.\ Max R\"{o}ssler, the Walter Haefner Foundation, and the ETH Z\"{u}rich Foundation. Email: {\tt{yuval.wigderson@eth-its.ethz.ch}}}}

\date{}

\begin{document}

\maketitle
\begin{abstract}
    The canonical Ramsey theorem of Erd\H os and Rado implies that for any graph $H$, any edge-coloring (with an arbitrary number of colors) of a sufficiently large complete graph $K_N$ contains a monochromatic, lexicographic, or rainbow copy of $H$. The least such $N$ is called the Erd\H os--Rado number of $H$, denoted by $ER(H)$.  Erd\H os--Rado numbers of cliques have received considerable attention, and in this paper we extend this line of research by studying Erd\H os--Rado numbers of sparse graphs.
    For example, we prove that if $H$ has bounded degree, then $ER(H)$ is polynomial in $\ab{V(H)}$ if $H$ is bipartite, but exponential in general. 
    
    We also study the closely-related problem of constrained Ramsey numbers. For a given tree $\cS$ and given path $P_t$, we study the minimum $N$ such that every edge-coloring of $K_N$ contains a monochromatic copy of $\cS$ or a rainbow copy of $P_t$. We prove a nearly optimal upper bound for this problem, which differs from the best known lower bound by a function of inverse-Ackermann \nolinebreak type.
\end{abstract}

\section{Introduction}
Ramsey's theorem asserts that if the edges of a sufficiently large complete graph $K_N$ are colored with a fixed palette of colors, then the coloring contains an arbitrarily large monochromatic clique. Moreover, it is easy to see that such a statement cannot be true if we allow the palette of colors to not be fixed; e.g.\ if we allow $\binom N2$ colors to be used then every edge may receive a distinct color, and then we cannot find any monochromatic structure containing more than one edge.

Nonetheless, one {can} prove meaningful statements about edge-colorings of complete graphs with no assumption on the palette of colors. The foundational result of this type is the canonical Ramsey theorem of Erd\H os and Rado \cite{ER50}. In order to state it, we introduce the following terminology. We say that vertices $v_1, \dots, v_n$ form a {\em canonically colored} copy of $K_n$ if one of the following three conditions holds:
\begin{itemize}
    \item all edges $v_iv_j$ have the same color (the clique on $v_1, \dots, v_n$ is \emph{monochromatic}),
    \item for every $i$, all edges $v_iv_j$ for $j>i$ have the same color, say $c_i$, and the colors $c_1, c_2, \dots, c_{n-1}$ are distinct (the clique on $v_1, \dots, v_n$ is \emph{lexicographically colored}), or
    \item all edges $v_iv_j$ have distinct colors (the clique on $v_1, \dots, v_n$ is \emph{rainbow}).
\end{itemize}

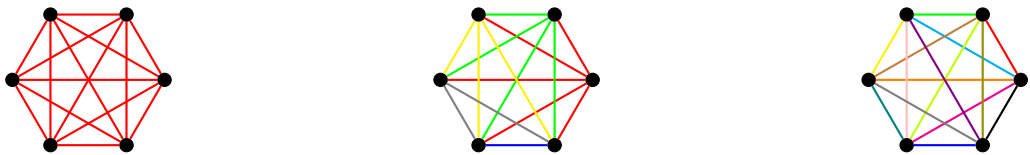
\begin{figure}[ht]
\begin{minipage}{0.33\textwidth}
\centering
	\begin{tikzpicture}
            \node[vertex] (v1) at (0:1) {};
            \node[vertex] (v2) at (60:1) {};
            \node[vertex] (v3) at (120:1) {};
            \node[vertex] (v4) at (180:1) {};
            \node[vertex] (v5) at (-120:1) {};
            \node[vertex] (v6) at (-60:1) {};
            \draw[thick, red] (v1) -- (v2) -- (v3) -- (v4) -- (v5) -- (v6) -- (v1);
            \draw[thick, red] (v1) -- (v3) -- (v5) -- (v1); 
            \draw[thick, red] (v2) -- (v4) -- (v6) -- (v2);
            \draw[thick, red] (v1) -- (v4);
            \draw[thick, red] (v2) -- (v5);
            \draw[thick, red] (v3) -- (v6);
        \end{tikzpicture}
\end{minipage}
\begin{minipage}{0.33\textwidth}
\centering
        \begin{tikzpicture}
            \node[vertex] (v1) at (0:1) {};
            \node[vertex] (v2) at (60:1) {};
            \node[vertex] (v3) at (120:1) {};
            \node[vertex] (v4) at (180:1) {};
            \node[vertex] (v5) at (-120:1) {};
            \node[vertex] (v6) at (-60:1) {};
            \draw[thick, red] (v1) -- (v2);
            \draw[thick, red] (v1) -- (v3);
            \draw[thick, red] (v1) -- (v4);
            \draw[thick, red] (v1) -- (v5);
            \draw[thick, red] (v1) -- (v6);
            \draw[thick, green] (v2) -- (v3);
            \draw[thick, green] (v2) -- (v4);
            \draw[thick, green] (v2) -- (v5);
            \draw[thick, green] (v2) -- (v6);
            \draw[thick, yellow] (v3) -- (v4);
            \draw[thick, yellow] (v3) -- (v5);
            \draw[thick, yellow] (v3) -- (v6);
            \draw[thick, gray] (v4) -- (v5);
            \draw[thick, gray] (v4) -- (v6);
            \draw[thick, blue] (v5) -- (v6);
        \end{tikzpicture}
\end{minipage}
\begin{minipage}{0.33\textwidth}
\centering
        \begin{tikzpicture}
           \node[vertex] (v1) at (0:1) {};
            \node[vertex] (v2) at (60:1) {};
            \node[vertex] (v3) at (120:1) {};
            \node[vertex] (v4) at (180:1) {};
            \node[vertex] (v5) at (-120:1) {};
            \node[vertex] (v6) at (-60:1) {};
            \draw[thick, red] (v1) -- (v2);
            \draw[thick, cyan] (v1) -- (v3);
            \draw[thick, orange] (v1) -- (v4);
            \draw[thick, magenta] (v1) -- (v5);
            \draw[thick, black] (v1) -- (v6);
            \draw[thick, green] (v2) -- (v3);
            \draw[thick, brown] (v2) -- (v4);
            \draw[thick, lime] (v2) -- (v5);
            \draw[thick, olive] (v2) -- (v6);
            \draw[thick, yellow] (v3) -- (v4);
            \draw[thick, pink] (v3) -- (v5);
            \draw[thick, violet] (v3) -- (v6);
            \draw[thick, teal] (v4) -- (v5);
            \draw[thick, gray] (v4) -- (v6);
            \draw[thick, blue] (v5) -- (v6);
        \end{tikzpicture}
\end{minipage}
\caption{Illustrations of canonical colorings of a six-vertex clique}
\end{figure}
The Erd\H os--Rado canonical Ramsey theorem then reads as follows.
\begin{theorem}[Erd\H os--Rado \cite{ER50}]\label{thm:ER}
    For every integer $n$, there exists some $N$ such that every edge-coloring of $E(K_N)$ (with an arbitrary number of colors) contains a canonically colored copy of $K_n$. 
\end{theorem}
Let us denote by $ER(K_n)$ the smallest integer $N$ for which every edge-coloring of $K_N$ contains a canonically colored copy of $K_n$. The original proof of Erd\H os and Rado \cite{ER50} was via a reduction to a $4$-uniform hypergraph Ramsey problem, and thus gave weak bounds of the form $ER(K_n) \leq 2^{2^{2^{O(n)}}}$. However, an alternative proof of \cref{thm:ER} was found by Lefmann and R\"odl \cite{LR95}, which in particular supplies the much stronger bound $ER(K_n)\leq n^{O(n^2)}$. In the other direction, it is easy to show that $ER(K_n)$ is at least the $(n-2)$-color Ramsey number of $K_n$, yielding a lower bound of $ER(K_n) \geq 2^{\Omega(n^2)}$ via a well-known lower bound of Abbott \cite{MR0314673} (rediscovered several times, e.g.\ \cite{MR0932230,MR0316290}) on multicolor Ramsey numbers. We thus know the behavior of $ER(K_n)$ up to a logarithmic gap in the exponent, and closing this gap is a major challenge.

In this paper, we are interested in studying the Erd\H os--Rado numbers of sparser graphs than cliques. Formally, let us define the Erd\H os--Rado number of a graph $H$ to be the smallest $N$ such that any edge-coloring of $K_N$ contains a {\em canonically colored} copy of $H$. Here, we say that a copy of $H$ is canonically colored if it is monochromatic, rainbow, or there is an ordering of its vertices $v_1, v_2, \dots$ such that for every $i$, the edges $v_iv_j\in E(H)$ where $j>i$ all have the same color, and these colors are distinct for different vertices $v_i$. Note that this quantity is well-defined, for if $H$ has $n$ vertices, then $ER(H) \leq ER(K_n) \leq n^{O(n^2)}$.

One of the most important discoveries in graph Ramsey theory is that the Ramsey numbers\footnote{Given a graph $H$, its \emph{Ramsey number} is the least $N$ such that every two-coloring of $E(K_N)$ contains a monochromatic copy of $H$.} of sparse graphs are much smaller than the corresponding Ramsey numbers of complete graphs. Thus, for example, a foundational result of Chvat\'al, R\"odl, Szemer\'edi, and Trotter \cite{MR714447} states that bounded-degree graphs have Ramsey numbers which are linear in their number of vertices (in contrast to cliques, whose Ramsey number is exponential). This result was extended to graphs of bounded degeneracy by Lee \cite{MR3664811}, confirming a famous conjecture of Burr and Erd\H os \cite{MR0371701}.

Recently, a number of authors have studied the extent to which analogous statements are true in other Ramsey-theoretic settings, such as for hypergraphs \cite{MR2532871,MR2520273,MR2410941, FSSTZ23}, vertex- and edge-ordered graphs \cite{MR4170446,MR3575208,MR4057168}, and directed graphs \cite{FHW,MR2793448}. In the canonical Ramsey setting, it is natural to wonder whether, for a ``sparse'' $n$-vertex graph $H$, its Erd\H os--Rado number $ER(H)$ is much smaller than $n^{O(n^2)}$, the generic bound which holds for all $n$-vertex graphs. In particular, if $H$ is sparse, is $ER(H)$ polynomial in $n$?

Our first result essentially resolves this question, obtaining nearly matching upper and lower bounds for $ER(H)$ for a wide class of sparse $H$. Interestingly, our results reveal that $ER(H)$ is polynomial in $n$ if and only if $H$ is bipartite, a curious condition that has no obvious analogue in the world of classical Ramsey numbers. For simplicity, we state the following theorem only for regular graphs, although it holds in much greater generality, as we discuss in \cref{sec:bounded deg}.
\begin{theorem}\label{thm:ER bounds}
    There exist absolute constants $C>c>0$ such that the following holds. Let $H$ be an $n$-vertex $d$-regular graph.
    \begin{enumerate}
        \item If $\chi(H)=2$, then
        \[
        n^{cd} \leq ER(H) \leq n^{Cd}.
        \]
        \item If $\chi\coloneqq \chi(H)\geq 3$, then
        \[
        2^{cn} \leq ER(H) \leq n^{Cd\chi n}.
        \]
    \end{enumerate}
\end{theorem}
In other words, a $d$-regular bipartite graph has Erd\H os--Rado number equal to $n^{\Theta(d)}$, and if $d,\chi$ are fixed, then a $d$-regular non-bipartite graph has Erd\H os--Rado number equal to $2^{\Theta(n)}$. The actual results we prove are somewhat more general than those stated in \cref{thm:ER bounds}; for example, we show that the upper bound of $n^{Cd}$ holds even if the bipartite graph $H$ is only assumed to be $d$-degenerate\footnote{A graph is said to be $d$-degenerate if its vertices can be ordered as $v_1,\dots,v_n$ such that each vertex is adjacent to at most $d$ vertices which precede it, i.e.\ each $v_i$ is adjacent to at most $d$ vertices $v_j$ with $j<i$.}, rather than $d$-regular. We defer the precise statements to \cref{sec:bounded deg}.

Although \cref{thm:ER bounds} describes the approximate growth rate of $ER(H)$ for sparse graphs, it would still be desirable to obtain more precise information for certain restricted classes. In particular, one case that has been well-studied is that of trees. Recall that every tree is $1$-degenerate, which implies that every tree has a vertex order according to which the lexicographic coloring is rainbow. As such, \cref{thm:ER} implies that for every tree $\cT$, there exists some $N$ such that every coloring of $E(K_N)$ contains a monochromatic or rainbow copy of $\cT$. More generally, given two trees $\cS,\cT$, one can define $f(\cS,\cT)$ to be the least $N$ such that every coloring of $E(K_N)$ contains a monochromatic copy of $\cS$ or a rainbow copy of $\cT$.

The study of the function $f(\cS,\cT)$ was initiated by Jamison, Jiang, and Ling \cite{MR1943103}, who termed this the \emph{constrained Ramsey number} of $\cS$ and $\cT$. They proved a number of results on $f(\cS,\cT)$, including the bounds
\[\Omega(st) \leq f(\cS,\cT) \leq O(st\diam(\cT))\]
whenever $\cS$ has $s$ vertices, $\cT$ has $t$ vertices, and $\diam(\cT)$ denotes the diameter of $\cT$. In particular, these results show that $f(\cS,\cT) = \Theta(st)$ whenever $\cT$ has bounded diameter (e.g.\ is a star), but the upper and lower bounds are off by a factor of $\Theta(t)$ in the worst case. Jamison, Jiang, and Ling \cite{MR1943103} asked to narrow the gap, and tentatively conjectured that the lower bound is closer to the truth. In particular, they asked to determine the correct dependence on $t$ in general, and asked about the case where $\cT = P_t$ is a path on $t$ vertices, since this is the case where their bounds are furthest apart.

For fixed $s$, the first question was resolved by Wagner \cite{MR2238050}, who proved that $f(\cS, P_t) = O(s^2 t)$ for all $\cS$. In particular, this shows that the correct dependence on $t$ is linear when $s$ is held constant, but does not improve on the bound of Jamison--Jiang--Ling \cite{MR1943103} when $s$ and $t$ are of the same order. On the other hand, major progress on the second question was made by Loh and Sudakov \cite{LS09}, who proved that $f(\cS, P_t) = O(st \log t)$, which matches the lower bound up to a logarithmic factor and which is much stronger than the upper bounds of \cite{MR1943103,MR2238050} when $s$ and $t$ tend to infinity at comparable rates. Our second main result is a further improvement over the result of Loh--Sudakov \cite{LS09}: we prove an upper bound on $f(\cS, P_t)/(st)$ that is of inverse Ackermann type. In order to state this result precisely, let us recall the definition of the inverse Ackermann hierarchy.
\begin{definition}
The function $\alpha_1$ is given by $\alpha_1(t)=\lceil t/2\rceil$. For each $k\geq 2$, we define $\alpha_k(t)$ inductively, to be the smallest number of times the function $\alpha_{k-1}$ must be applied to $t$ so that the output becomes $1$. More formally, we have $\alpha_k(1)=0$ and $\alpha_k(t)=\min\{r: \alpha_{k-1}^{(r)}(t)=1\}$, where $\alpha_{k-1}^{(r)}$ denotes the $r$-fold application of the function $\alpha_{k-1}$.
\end{definition}

For example, we have $\alpha_2(n)=\lceil \log_2 n\rceil$ and $\alpha_3(n)=\log_\star n$. 
With this notation the result of Loh--Sudakov can be stated as saying that $f(\cS, P_t) = O(st \alpha_2(t))$, and our next main result obtains such a bound at every level of the inverse Ackermann hierarchy.

\begin{theorem}\label{thm:main}
For every integer $k$, there exists a constant $A_k$ such that the following holds. For every $s$-vertex tree $\cS$ and every integer $t\geq 2$, we have $f(\cS, P_t)\leq A_k s t \alpha_k(t)$.
\end{theorem}
This result suggests that the true value of $f(S,P_t)$ is $\Theta(st)$, and more generally supports the conjecture of Jamison--Jiang--Ling that the same bound holds for all trees $S,T$. Moreover, this result yields a nearly quadratic upper bound on the Erd\H os--Rado numbers of paths.
\begin{corollary}
    For all integers $k,t$, we have $ER(P_t) = O_k(t^2 \alpha_k(t))$.
\end{corollary}

The rest of this paper is organized as follows. We present our bounds on $ER(H)$ for sparse $H$ in \cref{sec:bounded deg}, and the proof of \cref{thm:main} in \cref{sec:ackermann}. The proofs in \cref{sec:bounded deg} are all fairly short and use well-known techniques in Ramsey theory, such as product colorings and the dependent random choice method. In contrast, the proof of \cref{thm:main} is rather involved, and is based on a complicated induction scheme based on an amortized version of the technique introduced by Loh and Sudakov \cite{LS09}.

\section{Erd\H{o}s--Rado numbers of bounded degree graphs}\label{sec:bounded deg}

We begin by stating the more precise versions of \cref{thm:ER bounds} that we will prove. We begin with the upper and lower bounds for bipartite graphs.

\begin{theorem}\label{thm:upper bound bipartite}
Let $H$ be a $t$-degenerate bipartite graph on $n$ vertices. Then $ER(H)\leq n^{O(t)}$.
\end{theorem}

\begin{theorem}\label{thm:lower bound bipartite}
Let $H$ be an $n$-vertex graph with average degree $d$. Then $ER(H) \geq n^{\Omega(d)}$.
\end{theorem}
Note that \cref{thm:lower bound bipartite} holds even for non-bipartite $H$. However, the following result yields a much stronger lower bound if $\chi(H)>2$ and the average degree of $H$ is proportional to its maximum degree.

\begin{theorem}\label{thm:lower bound non-bipartite}
    Let $H$ be an $n$-vertex graph with maximum degree $\Delta$ and average degree $d$, and suppose that $\chi(H) \geq 3$. Then $ER(H) > 2^{\lceil nd/(2\Delta)\rceil -1}$.
\end{theorem}
In particular, if $d = \Theta(\Delta)$ and $\chi(H) \geq 3$, then $ER(H) \geq 2^{\Omega(n)}$. We remark that the assumption that the maximum degree is not much larger than the average degree is necessary in order to get an exponential lower bound. Indeed, it is not hard to show that if $H_n$ is the $n$-vertex graph obtained from $K_{1,n-1}$ by adding an edge between two leaves, then $\chi(H_n)=3$ but $ER(K_n) \leq (n-1)^2$.

In the other direction, we prove the following exponential upper bound on $ER(H)$ for non-bipartite $H$.

\begin{theorem}\label{thm:upper bound non-bipartite}
Let $H$ be a $n$-vertex graph of maximum degree $\Delta \geq 2$ and chromatic number $\chi=\chi(H)\geq 3$. Then $ER(H)\leq n^{O(\Delta\chi n)}$.
\end{theorem}

\subsection{Lower bounds}

We begin by proving the general lower bound on $ER(H)$ from the Theorem~\ref{thm:lower bound bipartite}, which turns out to be tight for bipartite $H$.

\begin{proof}[Proof of Theorem~\ref{thm:lower bound bipartite}.]
As $ER(H) \geq n$, the result is trivially true if $d \leq 4$, so we assume henceforth that $d > 4$. We will consider a random coloring of $E(K_N)$ using $n$ colors, where $N=n^{(d-4)/2}$. As $H$ has more than $n$ edges, since $e(H)=nd/2>n$, there is no rainbow copy of $H$ in such a coloring. We now estimate the probability that there is a monochromatic or lexicographic copy \nolinebreak of \nolinebreak $H$.

There are $n!$ ways of ordering the vertices of $H$ as $v_1,\dots,v_n$, and then at most $N^n$ ways of picking $u_1,\dots,u_n \in V(K_N)$. For such a choice, the probability that it defines a monochromatic or lexicographic copy of $H$ (according to the given ordering of $H$) can be upper-bounded as follows. Let $N^+(v_i)$ denote the set of $j$ such that $v_i v_j \in E(H)$ and $j>i$, and let $d^+(v_i)=\ab{N^+(v_i)}$. If $u_1,\dots,u_n$ define a monochromatic or lexicographic copy of $H$, then we have that all edges between $u_i$ and $\{u_j: j \in N^+(v_i)\}$ have the same color. The probability that this happens is exactly $n^{1-d^+(v_i)}$. Therefore, the probability that $u_1,\dots,u_n$ is a monochromatic or lexicographic copy of $H$ is at most
    \[
    \prod_{i=1}^n n^{1-d^+(v_i)} = n^n \cdot n^{-\sum_{i=1}^n d^+(v_i)} = n^n \cdot n^{-dn/2} = n^{-(d-2)n/2} .
    \]
    By the union bound, the probability of finding any monochromatic or lexicographic copy of $H$ is at most
    \[
        n! \cdot N^n \cdot n^{-(d-2)n/2} < \left( \frac{nN}{n^{(d-2)/2}}\right)^n = \left( \frac{N}{n^{(d-4)/2}}\right)^n =1.\qedhere
    \]
Hence, there exists a coloring of $K_N$ which contains no canonical copy of $H$. 
\end{proof}
\begin{remark}
This proof actually implies a slightly stronger statement, i.e. that there exists an edge-coloring of the complete graph on $N=n^{\Omega(d)}$ vertices without a rainbow or a \textit{weakly lexicographic} copy of $H$. Here, a weakly-lexicographic copy of $H$ is a $n$-tuple of vertices $v_1, \dots, v_n$ where for each $i$ the edges $\{v_iv_j\mid j>i\}$ have the same color, but these colors are not necessarily \nolinebreak all \nolinebreak different.
\end{remark}

We now turn to \cref{thm:lower bound non-bipartite}, which supplies a much stronger lower bound if $H$ is not bipartite. 
Before giving this proof, we need the following simple lemma.

\begin{lemma}\label{lem:many color lex}
    Let $H$ be an $n$-vertex graph with average degree $d$ and maximum degree $\Delta$. Then in any lexicographic coloring of $H$, at least $nd/(2\Delta)$ colors are used.
\end{lemma}
\begin{proof}
    Fix an ordering of the vertices of $H$ as $v_1,\dots,v_n$, and consider the lexicographic coloring associated to this ordering. The number of colors used is equal to the number of vertices with positive forward degree, i.e.\ the number of vertices $v_i$ which have a neighbor $v_j$ with $j>i$. Let $S$ be the set of vertices with positive forward degree. By definition, $S$ is a vertex cover of $H$, i.e.\ every edge is incident to at least one vertex in $S$. Indeed, the left endpoint of every edge has positive forward degree, and is thus in $S$.

    The total number of edges incident to $S$ is at most $\ab S \Delta$, but is also equal to $nd/2$, the total number of edges in $H$. Hence $\ab S \geq nd/(2\Delta)$, as claimed.
\end{proof}
We remark that this proof actually shows that in any lexicographic coloring of $H$, at least $\tau(H)$ colors are used, where $\tau(H)$ denotes the vertex cover number of $H$. We then combine this with the simple lower bound $\tau(H) \geq nd/(2\Delta)$. In particular, one can strengthen \cref{thm:lower bound non-bipartite} to say that $ER(H) \geq 2^{\tau(H)-1}$ when $H$ is non-bipartite.

Given \cref{lem:many color lex}, we can prove \cref{thm:lower bound non-bipartite}.

\begin{proof}[Proof of Theorem~\ref{thm:lower bound non-bipartite}.]
    Let $r=\lceil \frac{nd}{2\Delta}\rceil -1$ and $N=2^r$. We have the standard hypercube coloring of $K_N$, where we identify $V(K_N)$ with $\{0,1\}^r$, and color every edge according to the first coordinate where its endpoints differ. The key property of this coloring is that every color class is bipartite.

    This immediately implies that this coloring contains no monochromatic copy of $H$, as $H$ is not bipartite. Additionally, the total number of colors used is $r<nd/(2\Delta)$, and thus by \cref{lem:many color lex}, there can be no lexicographic copy of $H$ in this coloring. 
    Finally, as $r < dn/2 = e(H)$, there is trivially no rainbow copy either.  
    This proves that $ER(H) \geq N$, as claimed.
\end{proof}

\subsection{Upper bounds}

In this section, we will prove upper bounds on the Erd\H{o}s--Rado numbers, both for bipartite and non-bipartite graphs. Let us begin with a folklore lemma which can be used to find rainbow copies of $H$ in a graph $G$ where no color appears at any vertex too many times (see e.g.\ \cite{MR2016871} for much more precise results). We give the proof for completeness.

\begin{lemma}\label{lemma:sparsity at every vertex implies rainbow copies}
Let $H$ be a graph with $|V(H)|=n$, and let $K_{ns}$ be an edge-colored complete graph on the vertex set $V_1\cup\dots\cup V_n$, where $|V_1|=\dots=|V_n|=s$. Suppose that for each pair $ij\in E(H)$, every vertex $v_i\in V_i$ sends at most $s/n^4$ edges of the same color to the set $V_j$. Then this coloring of $K_{ns}$ contains a rainbow copy of $H$.
\end{lemma}
\begin{proof}
Will prove that with positive probability, picking one vertex uniformly at random from each set $V_i$ gives a rainbow copy of $H$. More precisely, let $v_i\in V_i$ be a uniformly random vertex. There are two types of bad events: that the edge $v_iv_j$ shares the color with the edge $v_iv_k$, for distinct indices $i, j, k\in [n]$, and that the edge $v_iv_j$ shares the color with the edge $v_kv_\ell$, for distinct indices $i, j, k, \ell\in [n]$. If none of these bad events occurs, the copy of $H$ formed by $v_1, \dots, v_n$ is rainbow.

It is not hard to see that for any collection of indices, the probability of this bad event is at most $1/n^4$. The reason for this is that, once all vertices but $v_j$ are revealed, there is at most $s/n^4$ options for $v_j$ for which the edge $v_iv_j$ gets a problematic color (i.e., the same color as $v_iv_k$ in the first case and as $v_kv_{\ell}$ in the second case). Also, note that the number of bad events is at most $n(n-1)(n-2)+n(n-1)(n-2)(n-3)<n^4$. Therefore, by the union bound, with positive probability no bad event occurs and $v_1, \dots, v_n$ form a rainbow copy of $H$, completing the proof.
\end{proof}

Let us now focus on the case when $H$ is bipartite. To prove the upper bounds we promised, we need the following lemma, which easily follows from classical estimates on the extremal numbers of degenerate bipartite graphs.

\begin{lemma}\label{lemma:turan for degenerate bipartite graphs}
Let $H$ be a $t$-degenerate bipartite graph on $n$ vertices. Then every $m \times m$ bipartite graph with $m \geq 2^{12t}n^{16t+2}$ vertices and with edge density at least $\frac 12 n^{-4}$ contains a copy of $H$.
\end{lemma}
\begin{proof}
This is a direct corollary of the classical theorem of Alon, Krivelevich, and Sudakov \cite[Theorem 3.5]{AKS03}, which states that $\ex(2m, H)\leq n^{\frac{1}{2t}}(2m)^{2-\frac{1}{4t}}$ for every $t$-degenerate graph $H$. 
We have $\frac 12 n^{-4} \cdot m^2 > n^{\frac{1}{2t}}(2m)^{2-\frac{1}{4t}}$ for $m \geq 2^{12t}n^{16t+2}$. 
Therefore, every $m \times m$ bipartite graph with density at least $\frac 12 n^{-4}$ contains a copy of $H$. 
\end{proof}

We are now ready to combine \cref{lemma:sparsity at every vertex implies rainbow copies,lemma:turan for degenerate bipartite graphs} to prove \cref{thm:upper bound bipartite}, which states that the Erd\H{o}s--Rado number of a $t$-degenerate bipartite graph $H$ on $n$ vertices satisfies $ER(H)\leq n^{O(t)}$.

\begin{proof}[Proof of \cref{thm:upper bound bipartite}]
We prove that under the assumptions of the theorem, we have $ER(H)\leq N \coloneqq 2^{24t+1}n^{32t+6}$. So suppose, for the sake of contradiction, that we are given an edge-coloring of the graph $K_N$ without a canonical copy of $H$.

Let us partition the vertex set of $K_N$ in an arbitrary way into sets $U_1, \dots, U_n$ of size $|U_1|=\dots=|U_n|=2^{24t+1}n^{32t+5}\eqqcolon 2s$
Furthermore, for each pair $ij\in E(H)$, let us define $X_{i, j}$ as the set of vertices in $U_i$ which send at least $s/n^{4}$ edges of the same color to $U_j$. 

The first observation is that if we have $|\bigcup_{j \in N_H(i)}X_{i, j}|\leq |U_i|/2$ for all $i\in [n]$, then $K_N$ contains a rainbow copy of $H$. To show this, for each $i$, select an arbitrary set $V_i\subseteq U_i\setminus \bigcup_{j \in N_H(i)}X_{i, j}$ of size $s$ and apply Lemma~\ref{lemma:sparsity at every vertex implies rainbow copies} to the subgraph induced on $\bigcup_{i=1}^n V_i$. By the definition of the sets $X_{i, j}$, we know that for every $ij\in E(H)$, each vertex $v_i\in V_i$ sends at most $s/n^4$ edges of the same color to $U_j$, and hence Lemma~\ref{lemma:sparsity at every vertex implies rainbow copies} guarantees the existence of the rainbow copy of $H$.

Henceforth, we may assume that $|\bigcup_{j \in N_H(i)}X_{i, j}|\geq |U_i|/2$ for some $i$. By averaging, there exists $j$ such $|X_{i, j}| \geq |U_i|/(2n) \geq 2^{24t}n^{32t+4}$. For each vertex $v\in X_{i,j}$ denote by $c(v)$ a color in which $v$ sends at least $s/n^4$ edges to $U_j$. 

A simple pigeonhole principle argument shows that there is a subset $Y\subseteq X_{i, j}$ of size at least $\sqrt{|X_{i, j}|}\geq 2^{12t}n^{16t+2}$ such that the colors $c(v)$ are either the same for all $v\in Y$ or different for all $v\in Y$. The reason for this is simple: if fewer than $\sqrt{|X_{i, j}|}$ different colors appear among $\{c(v)\mid v\in X_{i, j}\}$, then there must be one color which is repeated at least $\frac{|X_{i, j}|}{\sqrt{|X_{i, j}|}}=\sqrt{|X_{i, j}|}$ times.

Let us now consider the bipartite graph between $Y$ and $U_j$, which contains, for all $v \in Y$, all edges of color $c(v)$ incident to $v$. This bipartite graph has at least 
$|Y|\cdot s/n^4=\frac{|Y||U_j|}{2n^4}$ edges, so its density is at least $\frac{1}{2n^4}$. By averaging, there exists a subset $Z\subseteq U_j$ of size $|Y|$ such that the density between $Y$ and $Z$ is at least $\frac 12 n^{-4}$. 
Since $|Z|=|Y|\geq 2^{12t}n^{16t+2}$, \cref{lemma:turan for degenerate bipartite graphs} guarantees that this bipartite graph contains a copy of $H$. 

If all vertices $v\in Y$ have the same $c(v)$, then this copy of $H$ is monochromatic, and if all the colors $c(v)$ are different, then this is a lexicographically colored copy of $H$. In any case, we have shown that a canonically colored copy of $H$ can be found in $K_N$, which completes the proof. 
\end{proof}

We now turn to the proof of \cref{thm:upper bound non-bipartite}, for which we will need the following lemma.

\begin{lemma}\label{lemma:one_step}
Let $H$ be an $n$-vertex graph of maximum degree $\Delta\geq 2$, let $K_s$ be an edge-colored complete graph on $s\geq (2n)^{7\Delta}$ vertices, and let $C$ be a set of at most $n$ colors. If $K_s$ does not contain a rainbow copy of $H$, then there exists a set $X\subseteq V(K_s)$ of size 
$|X|\geq (8n^6)^{-\Delta} \cdot s$ and a color $c$ such that either:
\begin{itemize}
    \item every $\Delta$-tuple of vertices in $X$ has at least $n$ common neighbors in color $c$, or
    \item $c\notin C$ and all vertices in $X$ have a common neighbor $v\notin X$ in color $c$.
\end{itemize}
\end{lemma}

\cref{lemma:one_step} is proved by an application of the dependent random choice method (see e.g.\ \cite{FS11} for an introduction to this method). For our purposes, the following simple lemma is sufficient.

\begin{lemma}[{\cite[Lemma 2.1]{FS11}}]\label{lemma:dep_rand_choice}
Let $\ell, m, k$ be positive integers. Let $G = (V, E)$ be a graph with $|V| = N$ vertices and average degree $d = 2|E(G)|/N$. If there is a positive integer $t$ such that
$$\frac{d^{t}}{N^{t-1}}-\binom{N}{k}\left(\frac{m}{N}\right)^{t}\geq \ell$$
then $G$ contains a subset $X$ of at least $\ell$ vertices such that every $k$ vertices in $X$ have at least $m$ common neighbors.
\end{lemma}

\begin{proof}[Proof of Lemma~\ref{lemma:one_step}]
For each color $c$, let us denote by $B_c$ the set of vertices which have at least $s/(2n^5)$ neighbors in the color $c$. If $B_c$ is nonempty for some $c\notin C$, then the second item in the lemma holds (and we are done), since we may choose this color $c$ and pick $X$ to be the neighborhood of a vertex which has $s/(2n^5)$ neighbors in $c$. So, we may assume henceforth that $B_c$ is empty when $c\notin C$.

We also claim that $\Big|\bigcup_{c\in C} B_c\Big|\geq s/2$. If this is not the case, pick a set $U\subseteq V(K_s)\setminus \bigcup_{c\in C} B_c$ of size $s/2$. No vertex of this set is incident to more than $|U|/n^5$ edges of the same color and therefore by partitioning $U$ arbitrarily into $n$ equal-sizes sets $U=U_1\cup\dots\cup U_n$, we conclude using Lemma~\ref{lemma:sparsity at every vertex implies rainbow copies} that $K_s$ contains a rainbow copy of $H$. Since this was assumed not to be the case, we must have $\Big|\bigcup_{c\in C} B_c\Big|\geq s/2$.

Since $|C| \leq n$, there must be a color $c$ satisfying $|B_c|\geq \frac{s}{2n}$. If we denote by $G$ the graph spanned by the edges of color $c$ in $K_s$, this implies that $G$ has at least $\frac{1}{2}|B_c| \cdot s/(2n^5) \geq s^2/(8n^6)$ edges and hence average degree 
$d \geq s/4n^6$. 
We will again choose this color $c$ and apply Lemma~\ref{lemma:dep_rand_choice} to $G$ to find a set $X$ of $\ell\coloneqq(8n^6)^{-\Delta} \cdot s$ vertices with the property that every $k\coloneqq\Delta$ vertices in $X$ have at least $m\coloneqq n$ common neighbors in color $c$. To this end, we pick $t=\Delta$ and verify that
\[\frac{d^\Delta}{s^{\Delta-1}}-\binom{s}{\Delta}\left(\frac{n}{s}\right)^{\Delta}\geq \frac{s}{(4n^6)^\Delta} - \frac{n^{\Delta}}{\Delta!}\geq 
\frac{1}{2}\cdot \frac{s}{(4n^6)^\Delta} \geq 
\frac{s}{(8n^6)^\Delta}, \]
since $s\geq (2n)^{7\Delta}$. This completes the proof.
\end{proof}

We now present the proof of Theorem~\ref{thm:upper bound non-bipartite}, which gives the upper bound $ER(H)\leq n^{O(\Delta \chi n)}$ for graphs of maximum degree $\Delta$ and chomatic number $\chi$.

\begin{proof}[Proof of Theorem~\ref{thm:upper bound non-bipartite}.]
We will show that $ER(H)\leq N \coloneqq (8n^6)^{\Delta \chi n}$. Assume we have an edge-coloring of $K_N$ with no canonically colored copy of $H$. Let us set $A_1=V(K_N)$ and iteratively apply Lemma~\ref{lemma:one_step} to obtain subsets $A_1\supseteq A_2\supseteq \cdots$ along with colors $c_1, c_2, \dots$ as follows.

Having defined $A_1, \dots, A_i$, we let $C_i$ be the set of colors $c_{j}$, $j< i$, for which there exists a vertex $v\in A_j$ adjacent to all of $A_{j+1}$ in color $j$; i.e., we consider the set of indices $j < i$ for which the second outcome in Lemma~\ref{lemma:one_step} holds, and let $C_i$ be the set of the corresponding colors $c_j$. Then, we apply Lemma~\ref{lemma:one_step} to the set $A_i$ with the color set $C_i$ to find a subset $A_{i+1}\subseteq A_i$ of size $|A_{i+1}|\geq (8n^6)^{-\Delta} |A_i|$ and the color $c_i$ with the following properties: either every $\Delta$-tuple of vertices of $A_{i+1}$ has $n$ common neighbors in $A_i$ in color $c_i$, or there is a vertex $v\in A_i$ adjacent to all of $A_{i+1}$ in color $c_i\in C_i$ (and in this case, we add $c_i$ to all the future sets $C_{i+1},C_{i+2},\dots$). In the first case, we say that step $i$ is a \textit{DRC-step}, while in the second case we say that step $i$ is a \textit{star-step}.

Note that the size of $A_i$ is at least 
$|A_i|\geq (8n)^{-6\Delta (i-1)} N>(2n)^{7\Delta}$ as long as $i\leq \chi n-1$ (by our choice of $N$), and hence we can define sets $A_1\supseteq\cdots\supseteq A_{\chi n}$.  

Suppose we have performed at least $n-1$ star-steps, say at indices $i_1, \dots, i_{n-1}$.
By the definition of a star-step, for each $k = 1,\dots,n-1$, there is $v_k \in A_{i_k}$ such that $v_k$ is connected in color $c_{i_k}$ to all vertices of $A_{i_k+1}$. Take also an arbitrary $v_n \in A_{\chi n}$.
The $n$-tuples of vertices $v_1, \dots, v_n$ forms a lexicographic copy of $K_n$. To see this, note that all edges from $v_k$ to $A_{i_k+1}$ are of color $c_{i_k}$, and all vertices $v_{k+1}, \dots, v_{n}$ are in $A_{i_k+1}$ since the sets $A_1, \dots, A_{\chi n}$ are nested. Moreover, the colors used for different star-steps are different (because by our choice of $C_i$ we always exclude the colors used in previous star steps), showing that $v_1, \dots, v_n$ is indeed a lexicographic copy of $K_n$. But then, it is a lexicographically colored copy of $H$ too, which we assumed does not exist in $K_s$.

Henceforth, we assume the number of star-steps is less than $n-1$. This means that at least $(\chi-1)(n-1)+1$ DRC-steps have been performed. The pigeonhole principle implies then that either $\chi$ of these DRC steps have the same color $c_i$ or $n$ steps have different colors $c_i$. Let us show now that in the first of these cases, one can find a monochromatic copy of $H$ in $K_s$, while in the second case one can find a lexicographic copy of $H$.  

Assume first there is a set of $\chi$ DRC-steps which all have the same color $c_i=c$, at indices $i_1, \dots, i_\chi$. Fix a $\chi$-coloring of $H$ with color classes $S_1, \dots, S_\chi$. We will construct an embedding $\phi$ of $H$ by embedding the vertices from the set $S_\chi$ first, and then working our way back to $S_1$, embedding $S_{\ell}$ into $A_{i_{\ell}}$ for each $\ell$.
To define the embedding $\phi$ on $S_\chi$, we can take an arbitrary injective map $\phi:S_\chi\to A_{i_\chi}$. Assuming we have embedded $S_{\chi}, \dots, S_{\ell+1}$ (for $\ell=\chi-1,\dots,1$), we define the embedding of a vertex $v\in S_\ell$ as follows: let $N_v \coloneqq N_H(v)\cap \bigcup_{j=\ell+1}^\chi S_j$ be the set of $\leq \Delta$ neighbors of $v$ in $\bigcup_{j=\ell+1}^\chi S_j$, and consider the set
$\phi(N_v)\subseteq A_{i_{\ell+1}}\subseteq A_{i_\ell+1}$. 
We choose $\phi(v)$ to be one of the common neighbors of $\phi(N_v)$ in color $c$ in $A_{i_\ell}$, which was not already used in the embedding.
This is possible because $\phi(N_v)$ has at least $n$ common neighbors in color $c$ in $A_{i_\ell}$.
Since all edges of the embedding are ensured to have the color $c$, we obtain a monochromatic copy of $H$.

If, on the other hand, we performed $n$ DRC-steps with different colors $c_i$ at indices $i_1, \dots, i_n$, the embedding of a lexicographic copy of $H$ is even easier. Ordering the vertices of $H$ arbitrarily as $v_1, \dots, v_n$, we construct an embedding $\phi$ of $H$ by embedding $v_k\in A_{i_k}$ going backwards from $k=n$ to $k=1$. When choosing where to embed $v_k$, we consider the set $\phi(N_H(v_k)\cap\{v_{k+1}, \dots, v_n\})$. Since this set contains at most $\Delta$ vertices and is contained in $A_{i_{k+1}}\subseteq A_{i_k+1}$, 
it has a common neighbor in color $c_{i_k}$ inside $A_{i_k}$. Choose $\phi(v_k)$ to be such a common neighbor. 
This embedding has the property that for each $k$, all edges of the form $v_kv_j$ with $j > k$ have the same color $c_{i_k}$. Since the colors $c_{i_k}$ are distinct, this is indeed a lexicographically colored copy of $H$.
\end{proof}

\section{Constrained Ramsey numbers of trees and paths}\label{sec:ackermann}

Throughout the section, we fix a tree $\cS$ on $s$ vertices and a path $P_t$ on $t$ vertices. Recall that $f(\cS, P_t)$ denotes the smallest integer $N$ such that every edge-coloring of $K_N$ contains either a monochromatic copy of $\cS$ or a rainbow copy of $P_t$. The main result of this section bounds $f(\cS, P_t)$ by $O(st\alpha_k(t))$, where $\alpha_k(t)$ denotes the function at the $k$-th level of the inverse Ackermann hierarchy, as defined in the introduction.

This theorem improves upon previous work of Loh and Sudakov \cite{LS09}, who showed that $f(\cS, P_t)\leq O(st \log(t))=O(st \alpha_2(t))$. Our proof follows their approach, with a couple of new twists on their ideas. Hence, let us first recall some of the results they used.

\subsection{Basic lemmas from previous work}

Fix an edge-coloring of the graph $K_N$ which has no monochromatic copy of $\cS$ or rainbow copy of $P_t$. The first basic observation is that every color has to be sparse on every subset of vertices, since it does not contain the tree $\cS$. More precisely, we have the following lemma.

\begin{lemma}[{\cite[Lemma 2.1]{LS09}}]\label{lemma:sparsity}
Let $X$ be a subset of vertices of $K_N$ and let $c$ be an arbitrary color. If an edge coloring of $K_N$ does not contain a monochromatic copy of $\cS$, the number of edges of color $c$ spanned by $X$ is at most $s|X|$.
\end{lemma}

A simple corollary of this lemma is that any $k$ colors $c_1, \dots, c_k$ span at most $ks|X|$ edges among the vertices of $X$. 

Beyond this simple lemma, the key step in the proof from \cite{LS09} is to use the assumption $K_N$ has no monochromatic $\cS$ or rainbow $P_t$ to construct the following useful substructure.

\begin{lemma}\label{lemma:substructure}
Suppose the edges of $K_N$ have been colored in such a way that $K_N$ has no
monochromatic copy of $\cS$ and no rainbow $t$-vertex path, where $N\geq 310 st$.
Then there exists a set $R$ of ``rogue colors", a subset $U \subseteq V (G)$ with a partition $U = U_1\cup \dots\cup U_r$, an association of a distinct color $c_i \notin R$ to each $U_i$, and an orientation of some of the edges of the induced subgraph $G[U]$, which satisfy the following properties:
\begin{enumerate}[label=(\roman*)]
    \item $|U| > N/10$, $|R| < t$, and each $|U_i| < 2s$.
    \item For any edge between vertices $x \in U_i$ and $y \in U_j$ with $i \neq  j$, if it is directed $x\to y$, its color is $c_i$, if it is directed $y\to x$, its color is $c_j$, and if it is undirected, its color is in $R$.
    \item For any pair of vertices $x \in U_i$ and $y \in U_j$ (where $i$ may equal $j$), there exist at least $t$ vertices $z \notin U$ such that the color of the edge $xz$ is $c_i$ and the color of $yz$ is $c_j$. 
\end{enumerate}
\end{lemma}

Throughout the paper, we often encounter the situation in which we need to ensure that each vertex coming from a certain set belongs to a different $U_i$. Hence, we will say that two vertices $u, v\in U$ \emph{conflict} if they lie in the same set $U_i$.

To understand why this structure might be helpful for building a rainbow path, in the next section we show that there exists a way to glue certain shorter paths in order to obtain a long rainbow path. 

\subsection{Rainbow collections of paths}

\begin{definition}
Let $\cP=\{P_1, \dots, P_k\}$ be a collection of disjoint paths in $U$. This collection is called \textit{rainbow} if the following three conditions hold:
\begin{itemize}
    \item no two vertices $u\in P_j, v\in P_{j'}$ (where $j$ may equal $j'$) conflict, and
    \item each $P_i$ is a directed path, except maybe for the first edge which may be rogue,
    \item for all $i\in \{2, \dots, k\}$, the first edge of $P_i$ is not directed forward and it has a rogue color $r_i$ which is different from the colors of all other first edges of $P_j$ for $j\neq i$. Similarly, if $P_1$ starts on a rogue edge, its color is also different from $r_2, \dots, r_{k}$. 
\end{itemize}
The total length of a collection of paths, denoted by $\ell(\cP)$, is simply the total number of vertices of paths in the collection, i.e. 
\[\ell(\cP)=\sum_{i=1}^k |P_i|.\]
Finally, each rainbow collection defines the set of \textit{conflicting vertices}, denoted by $C(\cP)$, which is simply the union of the sets $U_i$ over all vertices $v$ belonging to a path in the collection $\cP$.
\end{definition}

The following lemma, which is implicit in \cite{LS09}, shows how to glue shorter paths from a rainbow collection $\cP$ in order to obtain a long rainbow path. We include the proof for completeness.

\begin{lemma}\label{lemma:building from directed paths}
If $\cP$ is a rainbow collection with $\ell(\cP)\geq t$, then there exists a rainbow path on at least $t$ vertices in this edge-coloring of $K_N$.
\end{lemma}
\begin{proof}
Let us denote by $u_iv_i$ the first edge of the path $P_i$, while $w_i$ stands for last vertex of this path. The idea is to connect together the paths $P_i$ and $P_{i+1}$ for all $i\in \{1, \dots, k-1\}$. 

Let us fix some $i$ and explain how to make the connection. We denote by $U_\ell$ and $U_{\ell'}$ the sets containing the vertices $w_i$ and $u_{i+1}$, i.e. $U_\ell=C(\{w_i\})$ and $U_{\ell'}=C(\{u_{i+1}\})$.

By the property $(iii)$ from Lemma~\ref{lemma:substructure}, we know that there are at least $t$ vertices $z_i\notin U$ for which $w_iz_i$ has color $c_\ell$ and $z_iu_{i+1}$ has color $c_{\ell'}$. In particular, this means that we can choose the vertices $z_1, z_2, \dots$ one by one, with the property that $z_i$ is different from $z_1, \dots, z_{i-1}$.

Then, we claim that the path $P=P_1z_1P_2z_2\dots P_{k-1}z_{k-1}P_k$ is a rainbow path of length at least $t$. Its length is $|P|=\sum_{i=1}^k |P_i|+k-1\geq t$. To show it is rainbow, we direct all the non-rogue edges on the path $P$ in the way they were directed in $U$, and we direct $u_{i+1}\to z_i, w_i\to z_i$.
We note that all vertices of the paths $P_i$ belong to different sets $U_i$ and therefore no directed edge emanating out of one of these vertices can share the color of any other such edge. On the other hand, the only edges which are not directed are the edges $u_iv_i$. But all of these edges have different rogue colors, which shows that no color on the path $P$ repeats.
\end{proof}

From now on, we will focus on the induced subgraph on $U$, where we will try to find a rainbow collection of length $\geq t$. Also, since no two vertices of the rainbow collection may come from the same set $U_i$, let us delete all edges which belong to the sets $U_i$.

\subsection{Special sequences of vertices}

To find a long rainbow collection of paths, we order the vertices of $U$ according to the median ordering, which is the ordering which maximizes the number of forward edges (in case there are multiple orderings with the same number of forward edges, we pick one arbitrarily). Since we will be working with the median ordering throughout the proof, we introduce the notation $[u, v]$ to denote the set of vertices between $u$ and $v$ in the median ordering, including $u$ and $v$. The open and half-open intervals $(u, v), [u, v)$ and $(u, v]$ are defined in a similar fashion. Also, an initial interval in the median ordering is an interval which starts at the first vertex. Finally, the median ordering we consider will be fixed and will not change throughout the proof.

A very important quantity in the proof will be the \textit{rogue degree} of a vertex $v\in U$. This is simply the number of rogue edges incident to $v$ and we denote it by $\deg_R v$. Further, for a subset $X\subseteq U$, we will denote by $\Delta_R(X)$ the maximum rogue degree of a vertex in the subgraph induced on $X$.

It will often be useful to control the rogue degrees of various vertices, and this is why we will maintain a set of \textit{bad vertices} $B$ (which one can think of as the vertices of high rogue degree), which we will be forbidden from using in the rainbow collection we are building.

Let us now explain how we will guarantee the third condition, that all but one path in a rainbow collection start from a rogue edge of a different color.

\begin{definition}\label{defn:special sequence}
Let $S\subseteq U$ be an initial interval in the median ordering and let $B\subseteq U$ be a set of bad vertices. A sequence of vertices $v_1, v_2, \dots, v_f\in S$ is called \textit{special} for a bad set $B$ if $v_1$ is the first vertex in the median ordering in $S\setminus B$ and the vertices $v_2, \dots, v_f$ are obtained by performing the following algorithm. For each $k$, we let $v_{k+1}\in S\setminus B$ be the first vertex in the median ordering which satisfies the following conditions:
\begin{enumerate}
    \item $v_{k+1}$ does not conflict with any of the vertices $v_1, \dots, v_k, u_2, \dots, u_{k}$.
    \item There exists another vertex $u_{k+1}\in U\setminus B$ such that 
    \begin{itemize}
    \item $u_{k+1}$ does not conflict with any of the vertices $v_1, \dots, v_k, v_{k+1}, u_2, \dots, u_{k}$, and
    \item the edge $v_{k+1}u_{k+1}$ has a rogue color different from all other colors of edges $v_2u_2, \dots, v_ku_k$.
    \end{itemize}
    
    If there are several such vertices $u_{k+1}$, choose one such that the index of the rogue color of the edge $u_{k+1}v_{k+1}$ is minimal.
\end{enumerate}
If such a vertex $v_{k+1}\in S\setminus B$ does not exist, we terminate the procedure and output the sequence $v_1, \dots, v_k$, setting $f=k$.
\end{definition}

An equivalent way of viewing the definition of the special sequence is as follows. Once the vertices $v_1, u_2, v_2, \dots, u_k, v_k$ have been defined, we can define a temporary bad set $B^{(k)}$ at step $k$ as the union of $B$ and $C(\{v_1, u_2v_2, \dots, u_kv_k\})$. Then, we choose $v_{k+1}$ to be the first vertex of $S\setminus B^{(k)}$ which is incident to a rogue edge of a new color in the subgraph induced on $U\setminus B^{(k)}$. 

The following lemma establishes the basic properties of a special sequence.

\begin{lemma}\label{lemma:properties of special sequences}
Let $S\subseteq U$ be an initial interval of $U$ in the median ordering and let $B\subseteq U$ be a set of bad vertices. Let $v_1,\dots,v_f$ be the special sequence defined by $S,B$, and let $u_2,\dots,u_f$ be the corresponding vertices from \cref{defn:special sequence}. 
The following three properties hold:
\begin{lemenum}
   \item No two vertices among $v_1, \dots, v_f, u_2, \dots, u_f$ are in the same set $U_i$.\label{lemit:i}
   \item $u_iv_i$ is an edge of rogue color $r_i$, different from all other colors of edges $u_jv_j$ for $j\neq i$.\label{lemit:ii}
   \item For each $i$, the vertices of $[v_1, v_i) \setminus \big(B\cup C(\{u_jv_j\mid j<i\})\big)$ are incident only to rogue edges of colors $r_2, \dots, r_{i-1}$ in the graph induced on $U\setminus\big(B\cup C(\{u_jv_j\mid j<i\})\big)$. Furthermore, the vertices of $S\setminus \big(B\cup C(\{u_jv_j\mid j\leq f\})\big)$ are only incident to rogue edges of colors $r_2, \dots, r_f$ in the graph induced on $U\setminus\big(B\cup C(\{u_jv_j\mid j<i\})\big)$. \label{lemit:iii}
   \item For all $i\geq 2$, the vertex $v_i$ comes before $u_i$. \label{lemit:iv}
\end{lemenum}
\end{lemma}
\begin{proof}
The properties (i) and (ii) are ensured directly by the algorithm in \cref{defn:special sequence}. Property (iii) is ensured since $v_{k+1}$ is picked to be the \textit{first} vertex of the median ordering satisfying conditions of the algorithm. Finally, to show property (iv), we observe that if $u_i$ comes before $v_i$, then $u_i\in [v_1, v_i) \setminus \big(B\cup C(\{u_jv_j|j<i\})\big)$ and it is incident to an edge of color $r_i\notin \{r_2, \dots, r_{i-1}\}$, contradicting property (iii).
\end{proof}

Properties (i)--(ii) imply that $\cP_0=\{v_1, u_2v_2, \dots, u_fv_f\}$ is a rainbow collection. 
We call it the {\em special rainbow collection}.
Recall also that we are assuming that there is no rainbow path of length $t$. 
Hence, we must have $f \leq t$ by \cref{lemma:building from directed paths}; namely, no special sequence contains more than $t$ elements.

The following lemma is an almost direct consequence of \cref{lemma:sparsity} and \cref{lemit:iii}, but we decide to state it explicitly for later use.

\begin{lemma}\label{lemma:rogue edges special sequence}
Let $v_1, \dots, v_f$ be a special sequence for a bad set $B_0$ in the interval $S$ and let $x, y$ be positive integers satisfying $1\leq x<y\leq f$. Then, the number of rogue edges in the set $(v_x, v_y]\setminus \Big(B_0\cup C(\{u_iv_i\mid i\leq y\})\Big)$ is at most $sy \big|(v_x, v_y]\big|$.
\end{lemma}
\begin{proof}
By \cref{lemit:iii}, the only rogue colors appearing in $(v_x, v_y]\setminus \Big(B_0\cup C(\{u_iv_i\mid i\leq y\})\Big)$ are $r_2, \dots, r_{y-1}$. \cref{lemma:sparsity} implies that for each $r_i$, there are at most $s\big|(v_x, v_y]\big|$ edges of this color in the set $(v_x, v_y]\setminus \Big(B_0\cup C(\{u_iv_i\mid i\leq y\})\Big)$, and hence there are at most $sy\big|(v_x, v_y]\big|$ rogue edges in this set.
\end{proof}

We now record two additional lemmas about special sequences, and how changing the sets $S$ and $B$ affects the special sequences. The following lemma shows that extending the initial interval $S$ simply extends the corresponding special sequence.

\begin{lemma}\label{lemma:special sequence on different intervals}
Let $S, S'$ be initial intervals, such that $S'\subseteq S$, and let $B_0\subseteq U$ be a set of bad vertices. Define two special sequences, $v_1, \dots, v_f$ with respect to $B_0$ and $S$, and $v_1', \dots, v'_{f'}$ with respect to $B_0$ and $S'$. Then $v_1', \dots, v'_{f'}$ is a prefix of $v_1, \dots, v_f$, consisting of those $v_i$ which belong to $S'$. 
\end{lemma}
\begin{proof}
The only place in which the initial interval $S$ enters the definition of a special sequence is to determine the point at which the defining algorithm terminates. Therefore, if $S'\subseteq S$, the algorithm defining $\{v_i'\}$ stops earlier than the algorithm defining $\{v_i\}$, and the output of these two algorithms must be the same up to this point. In other words, $v_1', \dots, v'_{f'}$ is a prefix of $v_1, \dots, v_{f}$. Moreover, if the sequence $v_1, \dots, v_{f}$ contained any other vertices in $S'$, these vertices would also be added to the sequence $v_1', \dots, v'_{f'}$, which therefore means $S\cap\{v_1, \dots, v_f\}=\{v_1', \dots, v'_{f'}\}$. 
\end{proof}

The next lemma gives us control on how the special sequence evolves when we change the bad set $B$. This will be used many times in the proof of \cref{thm:main}, as we will need to iteratively enlarge the bad set as we add new vertices to our rainbow collection of paths.

\begin{lemma}\label{lemma:unchanged special sequence}
Let $S$ be an initial interval, let $v_1, \dots, v_f$ be a special sequence defined with respect to a bad set $B$ and $S$, and let $I$ be the set of vertices of $S$ coming after $v_p$, for some $p\leq f$.
Also, assume $B'$ is a larger set of bad vertices with the following three properties:
\begin{itemize}
    \item $B$ is a subset of $B'$;
    \item The vertices $v_1, u_2, v_2, \dots, u_p, v_p$ do not belong to $B'$;
    \item If $v_i\in I$, then $C(v_i), C(u_i)\subseteq B'$. 
\end{itemize}
Then the special sequence $v_1', \dots, v_{f'}'$ defined with respect to the set $B'$ satisfies $v_i'=v_i$ for all $i\leq p$. Furthermore, the sequence satisfies that for all $v_i\in I$, with $i\leq \min\{f', f\}$, the vertex $v_i'$ does not come before $v_i$. 
\end{lemma}
\begin{proof}
Let us begin by proving that $v_i=v_i'$ for all $1\leq i\leq p$ and $u_i=u_i'$ for all $2\leq i\leq p$, by induction on $i$. For $i=1$, we know that $v_1$ is the first vertex outside $B$ in the median ordering. Since $v_1\notin B'$ and $B\subseteq B'$, $v_1$ must be the first vertex outside $B'$ in the median ordering too, meaning that $v_1'=v_1$.

For $1\leq i\leq p-1$, the vertex $v_{i+1}$ is the first vertex outside the temporary bad set $B^{(i)}=B\cup C(\{v_1, u_2v_2, \dots, u_iv_i\})$ which is incident to a rogue edge of color $\neq r_2, \dots, r_i$ in $U\setminus B^{(i)}$, and $v_{i+1}'$ is the first vertex outside the set $B'^{(i)}=B'\cup C(\{v_1', u_2'v_2', \dots, u_i'v_i'\})$ incident to a rogue edge of color $\neq r_2', \dots, r_i'$ in $U\setminus B'^{(i)}$. 
By the induction hypothesis, we know that $v_1=v_1', u_2=u_2', v_2=v_2', \dots, v_i=v_i'$, and so we also have the equality of the forbidden rogue colors $r_2=r_2', \dots, r_i=r_i'$.  
We also have that $C(\{v_1', u_2'v_2', \dots, u_i'v_i'\})=C(\{v_1, u_2v_2, \dots, u_iv_i\})$ and therefore $B^{(i)}\subseteq B'^{(i)}$ (since $B\subseteq B'$). Moreover, this also shows $u_{i+1}, v_{i+1}\notin B'^{(i)}$, since $u_{i+1}, v_{i+1}\notin B', C(\{v_1, u_2v_2, \dots, u_iv_i\})$. Hence, $v_{i+1}$ is indeed the first vertex outside the set $B'^{(i)}$ which is incident to a rogue edge of color $\neq r_2', \dots, r_i'$ in $U\setminus B'^{(i)}$, which shows that $v_{i+1}=v_{i+1}'$. Since $u_{i+1}\notin B'^{(i)}$ also, it follows that $u_{i+1}=u_{i+1}'$, which completes the induction and shows the first part of the statement.

Now, we focus on showing the second part of the statement. Before we begin, observe that $B^{(i)}\subseteq B'^{(p)}$ for all $i\leq f$. By unpacking the definition, this is equivalent to showing that $B\cup C(\{v_1, u_2v_2, \dots, u_iv_i\})\subseteq B'\cup C(\{v_1', \dots, v_{p}'u_p'\})$. By assumption, the set $B$ and the sets $C(\{u_jv_j\}$ for all $j\geq p+1$ are subsets of $B'$. Also, $C(\{u_jv_j\})\subseteq C(\{v_1', \dots, u_p'v_p'\}$ for $j\leq p$, since $u_jv_j=u_j'v_j'$, which suffices to show the observation. 

Suppose now for a contradiction that $v_{k+1}'$ comes before $v_{k+1}\in I$ for some $p<k\leq \min\{f, f'\}$, and let us consider the minimal such $k$. Among the colors $r_{p+1}', \dots, r_{k+1}'$ there are $k-p+1$ colors and so there exists one which does not appear in the set $\{r_{p+1}, \dots, r_{k}\}$. Let this be the color $r_j'$. Since $v_j'\notin {B'}^{(j-1)}$ implies $v_j'\notin B^{(f)}$, we have that $v_j' \notin \{v_{1}, \dots, v_f\}$.

Consider the vertex $v_j'$ now: why is it not a part of the special sequence $v_1, \dots, v_f$? 
Suppose that $v_j'$ lies between $v_m$ and $v_{m+1}$ (for some $m \geq p$). Then $v_j'$ is incident to an edge of a rogue color $r_j'$, which is distinct from $r_1, \dots, r_m$. The only reason why $v_j'$ would not be a part of the original sequence is then that either $v_j'$ or $u_j'$ is contained in the set $B^{(m)}$. But observe that $B^{(m)}\subseteq B'^{(p)}\subseteq B'^{(j-1)}$, as we noted above, and we have $u_j', v_j'\notin B'^{(j-1)}$. This is a contradiction, which completes the proof of the lemma.
\end{proof}

The special vertices $v_1, \dots, v_f$ played an important role in the proof of Loh and Sudakov, who showed the following lemma.

\begin{lemma}[{\cite[Lemma 2.5]{LS09}}]\label{lemma:long I_k}
Let $B$ be the set of vertices in $U$ whose rogue degree is at least $4st$, and let $\{v_i\}_{i=1}^f$ be the sequence of special vertices defined with respect to it and the interval $S$ containing all vertices. If there exists an index $k\in [f]$ for which the interval $[v_{k}, v_{2k}]\setminus B$ contains at least $176st$ vertices, then there exists a rainbow collection of paths of total length at least $t$.
\end{lemma}

In the case $k\geq f/2$, in the above lemma $[v_k, v_{2k}]$ just stands for the set of vertices after $v_k$ in the median ordering. From now on, in case $j>f$, the interval $[v_i, v_j]$ will stand for the set of vertices coming after $v_i$ in the median ordering, and similarly for $(v_i,v_j)$.

With this lemma in mind, it is clear how Loh--Sudakov obtain the bound $f(\cS, P_t)\leq O(st\log t)$, since if $|U\setminus B|\geq Cst\log t$ for some large constant $C$, there will exist an index $k$ for which $|[v_k, v_{2k}]\setminus B|\geq |U\setminus B|/\log t\geq 176 st$. 

\subsection{Amortized basic lemma}

We begin by generalizing \cref{lemma:long I_k} very slightly, to obtain an amortized version of it, which will be the basic building block in our inductive scheme. Let us recall a piece of notation: for $X\subseteq U$, $\Delta_R(X)$ stands for the maximum rogue degree of a vertex in the subgraph induced on $X$. 

\begin{lemma}\label{lemma:tournaments}
Let $v_1, \dots, v_f$ be a sequence of special vertices, defined with respect to a set $B_0$ on some initial interval of $U$, and let $B=B_0\cup C(\{v_1, u_2v_2, \dots, u_fv_f\})$. Let $\ell\leq f$ be a positive integer and let $I$ be the interval $I=(v_\ell, x]$, for some\footnote{Recalling the notation introduced above, we have that $(v_\ell, v_{8\ell}]$ comprises all vertices coming after $v_\ell$ in case $v_{8\ell}$ is not defined. Thus, if $8\ell>f$, we are only assuming that $x$ comes after $v_\ell$.} $x\in (v_\ell, v_{8\ell}]$.

Let $\cP$ be a rainbow collection of paths all of whose vertices lie in $[v_1, v_\ell]$. Further, suppose there exist $\ell/2$ paths $P_1, \dots, P_{\ell/2}$ in $\cP$ whose endpoints $w_1, \dots, w_{\ell/2}$ (indexed according to the median ordering) have the property that $\Delta_R\big([w_1, x]\big)\leq |I|/10$. Furthermore, assume that 
\begin{equation}\label{eq:basic lemma condition}
|I|> 40\max\{|B\cap [w_1, x]|, \Delta_R(I\setminus B), s\}.
\end{equation}
Then there exists a collection of rainbow paths $\cP_1$ such that 
\begin{itemize}
    \item[(i)] every path in $\cP_1$ is an extension of a path from $\cP$, 
    \item[(ii)] for every every path $P\in \cP_1$ we have $V(P)\subseteq [w_1, x)\setminus B$, and
    \item[(iii)] $\ell(\cP_1)\geq \ell(\cP)+\Omega(|I|/s)$.
\end{itemize}
\end{lemma}

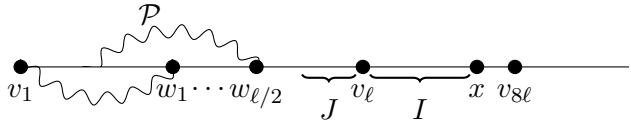
\begin{figure}[ht]
    \begin{center}
	\begin{tikzpicture}
            \draw[ black] (-5, 0) -- (3, 0);
            \node[vertex,label=below:$v_1$] (v1) at (-5,0) {};
            \node[vertex,label=below:$w_1$] (w2) at (-3,0) {};
            \node[,label=below:$\dots$] (d) at (-2.5,0) {};
            \node[vertex,label=below:$w_{\ell/2}$] (wl/2) at (-1.9,0) {};
            \node[,label=above:$\cP$] (P) at (-3.3,0.3) {};

            \begin{scope}[shift={(-1.55,0)}]
                \clip(-5,0) rectangle (2,1);
                \draw[decorate,decoration={snake}] (-90 : 1.5) arc (-15 : 190 : 1.5);
            \end{scope}
            \begin{scope}[shift={(-2.7,0.1)}]
                \clip(-5,0) rectangle (2,-0.75);
                \draw[decorate,decoration={snake}] (90 : 1.2) arc (15 : -180 : 1.3);
            \end{scope}
            
            \node[vertex,label=below:$v_\ell$] (vk) at (-0.5, 0) {};
            \node[vertex,label=below:$x$] (u) at (1,0) {};
            \node[vertex,label=below:$v_{8\ell}$] (v8k) at (1.5,0) {};
            \draw [ thick, decoration={ brace, mirror, raise=0.1cm}, decorate] (vk) -- (u) 
                node [pos=0.5, anchor=north, yshift=-0cm, label=below:$I$] {}; 
            \draw [ thick, decoration={ brace, mirror, raise=0.1cm}, decorate] (-1.3, 0) -- (vk) 
                node [pos=0.5, anchor=north, yshift=-0cm, label=below:$J$] {};
        \end{tikzpicture}
    \end{center}
    \caption{Illustration of the setup of Lemma~\ref{lemma:tournaments}.}
    \label{fig1}
\end{figure}
\begin{proof}
The strategy of the proof will be the following: if we extend the $\ell/2$ paths $P_1, \dots, P_{\ell/2}$ to longer directed paths, while maintaining the property that no two vertices of the extended collection conflict, then the extended collection will still be rainbow. Hence, we do this extension procedure in two steps. In the first step, we show that the paths $P_1, \dots, P_{\ell/2}$ can be extended using the vertices of the interval $[w_1, v_{\ell})\setminus B$ such that they end in the $3|I|/10$ vertices preceding $v_\ell$. Then, in the second step we construct a small number of tournaments using a constant fraction of the vertices of $I$ and use them to extend the paths $P_i$. 

\paragraph{Step 1:} Let $r=3|I|/10$, and let $J$ be the set of vertices not in $B$ among the $r$ vertices preceding $v_\ell$. We will show that $P_1, \dots, P_{\ell/2}$ can be extended to paths $P_1', \dots, P_{\ell/2}'$ which end in $J$ and such that the extended collection of paths $\cP'$ is still rainbow.

The initial observation is that for every vertex $v\in [w_1, v_\ell)\setminus B$, there exists a forward edge $v\to v'$ of length at most $r$ to a vertex $v'\notin B$ (the length of the edge $v\to v'$ is equal to one plus the number of vertices between $v$ and $v'$). To show this, note that in the median ordering, every vertex $v$ has at least as many forward as backward edges to the next $r$ vertices; if this was not the case, moving $v$ back $r$ places in the median ordering would increase the number of forward edges in the median ordering. 
Now, consider the edges between $v$ and the $r$ vertices following it in the median ordering. At most $\Delta_R([w_1,x)) \leq \frac{|I|}{10}$ of these edges are rogue. Also, at most $2s$ of them connect $v$ to a vertex in $C(\{v\})$ (i.e., a vertex from the same set $U_i$ as $v$). All other $r - \frac{|I|}{10} - 2s \geq \frac{3|I|}{20}$ edges are directed, and among them, at least half, so at least $\frac{3|I|}{40}$, are directed forwards.  
As long as the number of bad vertices in $[w_1, x]$ is less than $3|I|/40$ (and it is by the assumption that $|B\cap [w_1, x]|\leq |I|/40$), one of these forward edges does not lead to a bad vertex. 
Hence, $v$ has a forward edge of length at most $r$ to a non-bad vertex.

With this tool in hand, we are able to extend the paths $P_1, \dots, P_{\ell/2}$ so that they end in $J$ using the following algorithm. As long as at least one of the paths $P_i$ ends at a vertex $v\notin J$, there is a forward edge from $v$ to some $v'\notin B$ among the next $r = 3|I|/10$ vertices. Then, we extend the path $P_i$ using the edge $v\to v'$ and add to the bad set all vertices conflicting with $v'$. The algorithm terminates either when at least $|I|/40$ new vertices have been added to the bad set or when the endpoints of all paths $P_1, \dots, P_{\ell/2}$ lie in $J$. Note that as long as the bad set contains less than $3|I|/40$ vertices in $[w_1, x)$, the forward edge $v\to v'$ can be found and therefore the algorithm will terminate in one of the two described ways. 

Let us denote by $B'$ the set of bad vertices when the algorithm terminates, and by $\cP'$ the extended collection of paths. If $|B'\setminus B|\geq |I|/40$, we have $\ell(\cP')\geq \ell(\cP)+\frac{|I|}{80 s}$, i.e.\ at least $\frac{|I|}{80 s}$ new vertices have been added to the collection of paths $\cP$. This is because adding a vertex to the collection $\cP'$ adds at most $2s$ vertices to the bad set and there is no other reason for which a vertex is declared bad in this algorithm. In this case, the collection of paths $\cP'$ already suffices to complete the proof, since we have ensured that no two vertices of the collection conflict.

If we have $|B'\setminus B|\leq |I|/40$, then the algorithm terminated when the paths $P_1, \dots, P_{\ell/2}$ have been successfully extended so that they end in the interval $J$.
This is because initially $|B| \leq \frac{|I|}{40}$ (by \eqref{eq:basic lemma condition}), and so the number of bad vertices does not exceed $\frac{|I|}{20}$ at any point of the algorithm.
Let us denote the newly obtained endpoints by $w_1', \dots, w'_{\ell/2}$. In this case, we have performed the first step and we can pass onto the second step of the proof.

\paragraph{Step 2:} In this step, we will extend the paths ending at $w_1', \dots, w_{\ell/2}'$ using the vertices of $I$. To simplify the calculations, let us form the sets $U_j'$ by taking unions of several sets $U_j\cap I\setminus B'$ such that $2s\leq |U_j'|< 6s$.

First, we discard each set $U_j'$ with probability $\frac{399}{400}$. Then, we form the sets $T_1, \dots, T_{\ell/2}$ in the following way. If a set $U_j'$ was not discarded, we choose a uniformly random vertex from $U_j'$ and put it into one of the sets $T_i$, uniformly at random. Then, we form the sets $T_i'$ by deleting from $T_i$ all vertices $v$ such that there is no directed edge from $w'_i$ to $v$. Since only the vertices of $T_i'$ will be used to extend the paths $P_1, \dots, P_{\ell/2}$, the last two steps ensure that all vertices of the extended collection come from different sets $U_i$.
Finally, we form the set $T_i''$ by deleting one vertex from each rogue edge with both endpoints in $T_i'$. 

By construction, every edge between two vertices of $T_i''$ is directed (i.e.\ $T_i''$ is a tournament), and hence it contains a Hamiltonian path. Furthermore, $w_i'$ has a directed forward edge to the start of this Hamiltonian path and therefore the path $P_i$ can be extended by at least $|T_i''|$ vertices. The conclusion is that we can extend the paths of $\cP'$ such that their total length is at least $\ell(\cP')+\sum_{i=1}^{\ell/2}|T_i''|$. Hence, to prove the lemma it suffices to show that on average, $\sum_{i=1}^{\ell/2}|T_i''|\geq \Omega(|I|/s)$.

Before we estimate $\E\big[|T_i''|\big]$, we need to estimate $\E\big[|T_i'|\big]$ and $\E\big[|T_i|\big]$. By the same argument as in the beginning of this proof, the number of forward edges from $w_i'$ to a vertex in $[w_i', x)$ is at least \[\frac{1}{2}\Big(|I|-\Delta_R\big([w_1, x)\big) - 2s\Big)\geq \frac{1}{2} \Big(|I|-\frac{|I|}{10}-\frac{|I|}{20}\Big)=\frac{17}{40}|I|.\]
Out of these edges, at most $\frac{3}{10}|I|$ point to the vertices before $v_\ell$ and at most $|B'\cap I|\leq |I|/20$ point to bad vertices. Hence, we conclude that $w_i'$ points to at least $3|I|/40$ vertices in the union of the sets $U_j'$.

Consider any fixed vertex $v$ in this union, and let $j$ such that $v \in U'_j$. The probability that $v$ is assigned to $T_i$ is $\Pb[v\in T_i]=\frac{1}{400 \cdot \ell/2\cdot |U_j'|}$. Hence, this probability is between $\frac{1}{1200  \ell s}$ and 
$\frac{1}{400  \ell s}$. Thus,  
\[\E\big[|T_i'|\big]\geq \frac{3|I|}{40} \times \frac{1}{1200\cdot \ell s}=\frac{|I|}{16\cdot 10^3 \cdot \ell s}.\]
To bound $\E\big[|T_i'|-|T_i''|\big]$, we observe that the number of deleted vertices from $T_i'$ is always at most the number of rogue edges with both ends in $T_i'$. 
Since there are at most $8\ell$ rogue colors appearing in $I\setminus B'$, \cref{lemma:sparsity} implies there are at most $8\ell s |I\setminus B'|$ rogue edges in $I\setminus B'$. 
Furthermore, since the events $v \in T_i'$ and $v' \in T_i'$ are independent for distinct $v,v'$, the probability that a rogue edge $vv'$ makes it into $T_i'$ is at most \[\Pb[v, v'\in T_i']\leq \left(\frac{1}{400 \ell s }\right)^2.\]
Thus, the expected number of rogue edges in $T_i'$ can be upper-bounded as 
\[\E\Big[|T_i'|-|T_i''|\Big]\leq 8\ell s |I|\left(\frac{1}{400 \ell s }\right)^2\leq \frac{|I|}{2\cdot 10^4 \ell s}.\]
Thus, summing over the $\ell/2$ sets $T_i''$ we see that 
\[\E\bigg[\sum_{i=1}^{\ell/2}|T_i''|\bigg]= \sum_{i=1}^{\ell/2}\E\big[|T_i'|\big] - \sum_{i=1}^{\ell/2}\E\big[|T_i'|-|T_i''|\big] \geq \frac{\ell}{2}\Big(\frac{|I|}{16\cdot 10^3 \ell s}-\frac{|I|}{2\cdot 10^4 \ell s}\Big)=\frac{|I|}{16\cdot 10^4 s}.\]
The conclusion is that $\ell(\cP_1)\geq \ell(\cP)+\frac{|I|}{16\cdot 10^4 s}$, for an appropriate choice of the sets $T_i''$, which suffices to complete the proof of the lemma.
\end{proof}

\subsection{Main induction}

The following proposition is the main technical part of the argument. It is a ``higher-order'' version of \cref{lemma:tournaments}, showing that if an interval $I$ is sufficiently long, then any rainbow collection of paths that end before it can be extended to a longer rainbow collection, where the added length is of order $\Omega(\ab I/s)$. Note that if we could prove such a statement with no assumptions on $I$, we could prove that $f(\cS, P_t)=O(st)$, by setting $I$ to be all of $U$. 

Of course, we are not able to prove such a statement, and there are assumptions on $I$ in \cref{prop:main induction}. Roughly speaking, the statement of \cref{prop:main induction} with some parameter $k$ is sufficient to prove the bound $f(\cS, P_t) = O(st \alpha_{k+1}(t))$. However, the statement of \cref{prop:main induction} includes a number of extra assumptions, which are designed to maintain an inductive approach: we prove the statement by induction on $k$, and the base case $k=1$ corresponds to \cref{lemma:tournaments}. Thus, for example, a single step of the inductive argument allows us to improve the Loh--Sudakov bound $f(\cS, P_t) = O(st\log t)$ to $f(\cS, P_t) = O(st \log_\star t)$. 

\begin{proposition}\label{prop:main induction}
Let $k$ be a fixed integer, let $a$ be a sufficiently large integer with respect to $k$, and let $t \geq a$.
\begin{itemize}
    \item Let $B_0$ be a set of bad vertices and let $v_1, \dots, v_{f}$ be the special sequence for $B_0$, on some initial interval $S$. Suppose $\cP$ is a rainbow collection of paths with start points $v_1, u_2, \dots, u_{t/a^3}$, such that all vertices except maybe the starts of paths of $\cP$ (i.e. the $u_i$'s and the $v_i$'s) lie in $[v_1, v_{t/a^3}]\setminus C(\{u_i, v_i\mid 1\leq i\leq f\})$.
    \item Let $x$ be a vertex in $(v_{t/a^3}, v_{t/\alpha_{k}(a)^3}]$ and let $I= (v_{t/a^3}, x]\subseteq S$. Let $B=B_0\cup C(\cP)\cup C(\{u_iv_i\mid v_i\in I\})$.
\end{itemize}
Suppose that $t/2a^3$ paths of the collection $\cP$ have endpoints $w_1, \dots, w_{t/2a^3}$, arranged in the median ordering. Assume that $\Delta_R\big([w_1, x]\setminus B\big)\leq \frac{|I|}{10^k}$ and that
\begin{equation}\label{eq:main proposition assumption}
|I|\geq C_k\big|B\cap [w_1, x]\big|+C_k \alpha_k(a) \Delta_R(I\setminus B)+C_k\frac{st}{\alpha_k(a)},
\end{equation}
with $C_k=40^k$.

Then it is possible to extend the paths of $\cP$ which end at $w_i$, and possibly construct some new paths using vertices of $I$, to form a rainbow collection of paths $\cP'$ which uses only vertices before $x$ and not from $B$ and satisfies $\ell(\cP')\geq \ell(\cP)+\Omega_k(\frac{|I|}{s})$.
\end{proposition}
\begin{proof}
We prove this statement by induction on $k$. The case $k=1$ follows by applying Lemma~\ref{lemma:tournaments} with $\ell=t/a^3$.

Hence, let $k\geq 2$ and let $r$ be the smallest integer for which $v_{t/\alpha_{k-1}^{(r)}(a)^3}$ is after $x$ in the median ordering. We will partition $I$ into intervals $J_1, \dots, J_{r}$ such that $J_\ell=(v_{t/\alpha_{k-1}^{(\ell-1)}(a)^3}, v_{t/\alpha_{k-1}^{(\ell)}(a)^3}]$ for $\ell=1, \dots, r-1$ and $J_r=(v_{t/\alpha_{k-1}^{(r-1)}(a)^3}, x]$. 

Under these definitions, we claim that $r< \alpha_k(a)$. Observe that by definition of $\alpha_k(a)$, we have $\alpha_{k-1}^{(\alpha_k(a))}(a)=1$. Since $v_{t/\alpha_{k-1}^{(r-1)}(a)^3}$ comes before $x$ and $x$ comes before $v_{t/\alpha_k(a)^3}$, we must have $\alpha_{k-1}^{(r-1)}(a)\geq \alpha_k(a)$, and so $\alpha_{k-1}^{(r)}(a)\geq \alpha_{k-1}(\alpha_k(a))>1$, where the last inequality holds as long as $a$ is large enough compared to $k$. The conclusion is that $\alpha_{k-1}^{(r)}(a)>1$ and so $r< \alpha_k(a)$. 
In other words, $I$ is partitioned into fewer than $\alpha_k(a)$ intervals.

The proof has two main cases. The first case when $|J_\ell|\geq |I|/10$ for some $\ell\in [r]$ and the second one where $|J_\ell|\leq |I|/10$ for all $\ell\in [r]$. In the first case, we proceed by applying the induction hypothesis to the longest interval $J_{\ell}$, while in the second case we apply the induction hypothesis to all intervals $J_{\ell}$ that are not very short.

\paragraph{Case 1:} Within this case, we have two subcases. Let $\ell$ be the index of the longest interval $J_\ell$, and let $b=\alpha_{k-1}^{(\ell-1)}(a)$, so that $J_\ell=(v_{t/b^3}, v_{t/\alpha_{k-1}(b)^3}]\subseteq I$ (or potentially $J_\ell=(v_{t/b^3}, x]$ if $\ell=r$).
Now, we split into two further subcases based on whether $\alpha_{k-1}(b)\geq \alpha_k(a)$ or $\alpha_{k-1}(b)<\alpha_k(a)$. Let us first consider the subcase where $\alpha_{k-1}(b)\geq \alpha_k(a)$.

\paragraph{Subcase 1.1:}
We begin by performing a cleaning procedure on $J_\ell$. Let $\beta=20C_{k-1}$ and let $H$ be the set of vertices of high rogue degree in $J_\ell$, defined as 
\[H\coloneqq\Big\{v\in J_\ell\setminus B ~\Big|~ \text{the rogue degree of $v$ in $J_\ell\setminus B$ is at least $\beta \frac{st}{\alpha_{k-1}(b)^3}$}\Big\} \setminus \{u_1, \dots, u_{t/b^3}\}.\] 

\begin{claim}\label{claim:upper bound on H}
There are at most $2|J_\ell|/\beta$ vertices in $H$.  
\end{claim}
\begin{proof}
By Lemma~\ref{lemma:rogue edges special sequence}, the number of rogue edges in $J_\ell\setminus B\subseteq (v_{t/b^3}, v_{t/\alpha_{k-1}(b)^3}]\setminus(B_0\cup C(\{u_iv_i\mid i\leq t/\alpha_{k-1}(b)^3\}))$ is at most $\frac{st}{\alpha_{k-1}(b)^3} |J_\ell|$. 
Hence, the set $H$ of vertices with rogue degree higher than $\beta \frac{st}{\alpha_{k-1}(b)^3}$ contains at most $2|J_\ell|/\beta$ vertices.
\end{proof}

Let us now define a new special sequence $v_1', \dots, v_{f'}'$ with respect to the initial interval ending at $v_{t/\alpha_{k-1}(b)^3}$ (which we denote by $S$) and the bad set of vertices 
\[B_0'=\left(B_0\cup H\cup \bigcup_{v_i\in J_\ell} C(\{u_i, v_i\})\cup (C(\cP)\cap J_\ell)\right)\setminus \{u_2, \dots, u_{t/b^3}\}.\] 

\begin{claim}\label{claim:claim 2 in subcase 1.1}
We have $v_i'=v_i$ for all $1\leq i\leq t/b^3$. For all $t/b^3 < i\leq \min\{f', t/\alpha_{k-1}(b)^3\}$, the vertex $v_i'$ does not come before $v_i$.
\end{claim}
\begin{proof}
By Lemma~\ref{lemma:special sequence on different intervals}, $v_1, \dots, v_{t/\alpha_{k-1}(b)^3}$ is the special sequence defined with respect to the initial interval $S$ and the bad set $B_0$. 

Hence, our goal is to apply Lemma~\ref{lemma:unchanged special sequence} to the sequence $v_1, \dots, v_{t/\alpha_{k-1}(b)^3}$, the bad sets $B_0$ and $B_0'$, and $p=t/b^3$. Hence, we verify the conditions of the lemma: $B_0\subseteq B_0'$ is clear. The vertices $v_2, \dots, u_{t/b^3}$ do not belong to $B_0'$ by definition and $v_1, ..., , v_{t/b^3}$ are not in $B_0'$ for the following reasons:

\begin{itemize}
    \item they are not in $B_0$ since they are in the special sequence defined with respect to $B_0$,
    \item they do not belong to $H\cup (C(\cP)\cap J_\ell)\subseteq J_\ell$ since $v_i$ come before $J_\ell$ for $i\leq t/b^3$,
    \item they do not belong to $\bigcup_{v_i\in J_\ell} C(\{u_i, v_i\})$ since $v_i, u_i\in J_\ell$ do not conflict with the previously defined $v_i'$s and $u_i'$s.
\end{itemize}
    Finally, we have defined $B_0'$ such that $C(u_i), C(v_i)\subseteq B_0'$ for all $u_i, v_i\in J_\ell$.

Therefore, Lemma~\ref{lemma:unchanged special sequence} applies and the claim follows directly.
\end{proof}

Our goal is now to apply the induction hypothesis to the interval $J_\ell$, with the parameters $k-1$ and $b$. Note that also $b$ can be made large enough compared to $k-1$ if $a$ is made sufficiently large compared to $k$, since $b\geq \alpha_{k-1}(b)\geq \alpha_k(a)$.

The starting set of bad vertices for the induction hypothesis will be $B_0'$, the special sequence will thus be $v_1', \dots, v_{f'}'$, and the rainbow collection $\cP'$ will be $\cP'=\cP\cup \{u_i'v_i'\mid t/a^3< i\leq t/b^3\}$. Note that this collection is indeed rainbow since the bad set used in defining the sequence $\{v_i'\}_{i=1}^{t/b^3}$ contains $C(\cP)\setminus \{v_1, \dots, u_{t/b^3}\}$, which ensures that no two vertices in $\cP'$ conflict. Further, to verify that the vertices of $\cP'$, except the vertices $u_i', v_i'$, do indeed come from the set $[v_1, v_{t/b^3}]\backslash C(\{u_i', v_i'|i\leq f'\})$, its sufficient to verify that the vertices of $\cP$ except the $u_i, v_i$ come from $[v_1, v_{t/b^3}]\backslash C(\{u_i', v_i'|i\leq f'\})$, which is true since the vertices $u_i', v_i'$ do not conflict with any vertices of $\cP$ besides the start of the paths.

By Claim~\ref{claim:claim 2 in subcase 1.1}, the vertex $v_{t/\alpha_{k-1}(b)^3}'$ (if it is defined) does not come before $v_{t/\alpha_{k-1}(b)^3}$, which is the end of $J_\ell$. 
We thus define a new set of bad vertices $B'=B_0'\cup C(\cP)\cup C(\{u_{i}'v_i'\mid v_i'\in J_\ell\})$, as in the statement of \cref{prop:main induction}.

The set of endpoints $w_i'$ will be the set of endpoints $w_i$ plus the set of $v_i'$ for $t/a^3< i\leq t/b^3$ (if there are more than $t/2b^3$ of them, we can pick an arbitrary subset). Note that $|J_\ell|/10^{k-1}\geq |I|/10^k$ and so the bound on the rogue degree of vertices in $[w_1, v_{t/\alpha_{k-1}(b)^3}]$ still holds since
\[\Delta_R\big([w_1, v_{t/\alpha_{k-1}(b)^3}]\setminus B'\big)\leq \Delta_R\big([w_1, x]\setminus B\big)\leq \frac{|I|}{10^k}\leq \frac{|J_\ell|}{10^{k-1}}.\]

Finally, the following claim verifies $J_\ell$ is long enough to apply the induction hypothesis.

\begin{claim}
The following inequality holds:
\begin{align}\label{eqn:induction hypothesis}
|J_\ell| \geq C_{k-1}\big| B'\cap [w_1,v_{t/\alpha_{k-1}(b)^3}] \big|+C_{k-1}\alpha_{k-1}(b)\Delta_R(J_\ell\setminus B')+C_{k-1} \frac{st}{\alpha_{k-1}(b)}.
\end{align}
\end{claim} 
\begin{proof}    
We prove the inequality in three steps, by bounding each of the terms separately. We start from the final term of (\ref{eqn:induction hypothesis}). Recalling the inequality $\alpha_{k-1}(b)\geq \alpha_k(a)$ which defines this subcase, we find that 
\begin{align}\label{eq:J ell LB}
C_{k-1} \frac{st}{\alpha_{k-1}(b)}\leq C_{k-1} \frac{st}{\alpha_{k}(a)}\leq \frac{C_{k-1}}{C_k} |I|\leq \frac{10 C_{k-1}}{C_{k}}|J_\ell|\leq \frac{|J_\ell|}{4},
\end{align}
where the second inequality in the above chain comes from the assumption of \cref{prop:main induction}, while the third inequality follows from the assumption $|J_\ell|\geq |I|/10$.

To bound the first term of (\ref{eqn:induction hypothesis}), we recall that
\[B'=B_0'\cup C(\cP)\cup C(\{u_{i}'v_i'\mid v_i'\in J_\ell\})= B_0\cup H\cup C(\cP)\cup C(\{u_{i}'v_i'\mid v_i'\in J_\ell\})\cup  C(\{u_iv_i\mid v_i\in J_\ell\}).\]
Hence, we observe that $B'\subseteq B\cup H \cup C(\{u_{i}'v_i'\mid v_i'\in J_\ell\})$ and so
\begin{align*}
\big|B'\cap [w_1,v_{t/\alpha_{k-1}(b)^3}]\big|&\leq \big|B\cap [w_1,v_{t/\alpha_{k-1}(b)^3}]\big|+|B'\setminus B|\\
&\leq \big|B \cap [w_1, x]\big|+|H|+\big|C(\{u_{i}'v_i'\mid v_i'\in J_\ell\})\big|\\
&\leq \frac{|I|}{C_k}+ \frac{2|J_\ell|}{\beta}+2s\cdot \frac{2t}{\alpha_{k-1}(b)^3},
\end{align*}
where we bounded the first term by the assumption of \cref{prop:main induction}, the second by \cref{claim:upper bound on H} and the third by the fact that there are at most $t/\alpha_{k-1}(b)^3$ vertices $v_i'$ defined in $J_\ell$, and each $u_i, v_i$ contributes at most $2s$ vertices to $C(\{u_{i}'v_i'\mid v_i'\in J_\ell\})$. 
From \eqref{eq:J ell LB} we have that $\frac{|J_{\ell}|}{C_{k-1}} \geq \frac{4st}{\alpha_{k-1}(b)}$, and so
\begin{align*}    
\big|B'\cap [w_1,v_{t/\alpha_{k-1}(b)^3}]\big|&\leq\frac{|I|}{C_k}+ \frac{2|J_\ell|}{\beta}+\frac{4st}{\alpha_{k-1}(b)^3}
\leq \left(\frac{10}{C_k}+\frac{2}{\beta}+\frac{C_{k-1}^{-1}}{\alpha_{k-1}(b)^2}\right)|J_\ell|.
\end{align*}
Finally, if $b$ is large enough compared to $k-1$ (which we recall we may assume), we have $\alpha_{k-1}(b)\geq 4$ and so 
\[C_{k-1}\big|B'\cap [w_1,v_{t/\alpha_{k-1}(b)^3}]\big|\leq C_{k-1} \left(\frac{1}{4C_{k-1}}+\frac{1}{10C_{k-1}}+\frac{1}{16C_{k-1}}\right)|J_\ell|< \frac{|J_{\ell}|}{2}.\]

We conclude by bounding the middle term of \eqref{eqn:induction hypothesis}. Since $H\subseteq B'$, no vertex of $J_\ell\setminus B'$ has rogue degree more than $\beta st/\alpha_{k-1}(b)^3$ in $J_\ell\setminus B'$, i.e.\ we have
$\Delta_R(J_\ell\setminus B')\leq \beta st/\alpha_{k-1}(b)^3$. Therefore, if $b$ is large enough compared to $k$ so that $\alpha_{k-1}(b)> 4\beta$, we have by \eqref{eq:J ell LB} that
\[C_{k-1}\alpha_{k-1}(b)\Delta_R(J_\ell\setminus B')\leq C_{k-1}\beta \frac{st}{\alpha_{k-1}(b)^2}\leq \beta\frac{1}{\alpha_{k-1}(b)}|J_\ell|< \frac{|J_\ell|}{4}.\]
In conclusion, we have shown that the first term of \eqref{eqn:induction hypothesis} is bounded by $|J_\ell|/2$, while the remaining two terms are bounded by $|J_\ell|/4$, thus showing that their sum is less than $|J_\ell|$ as claimed.
\end{proof}

Having proven (\ref{eqn:induction hypothesis}), we have verified all the assumptions of the induction hypothesis for $k-1$ and therefore we may apply it to $J_\ell$. We conclude that there exists a collection of paths $\cP''$ extending $\cP'$ which satisfies $\ell(\cP'')\geq \ell(\cP')+\Omega_{k-1}(\frac{|J_\ell|}{s})\geq \ell(\cP)+\Omega_k(\frac{|I|}{s})$, which suffices to complete the proof in this subcase. Note that the only paths of $\cP$ which get extended are those whose endpoints were among the vertices $w_i'$, which are precisely those paths which end in some $w_i$. 
Finally, since $B\subseteq B'$ and the vertices of $\cP''$ come before $x$ and are not from $B'$, they are also not from $B$. This completes the discussion of Subcase 1.1.

\paragraph{Subcase 1.2:}
Recall that in this subcase, we have $\alpha_{k-1}(b)<\alpha_k(a)$. Note that this can only happen if $J_r$ is the longest interval, i.e.\ $\ell=r$. Indeed, if $\ell \leq r-1$ we have $\alpha_{k-1}(b)=\alpha_{k-1}^{(\ell)}(a)\geq \alpha_{k-1}^{(r-1)}(a)$. But by definition of $r$, $\alpha_{k-1}^{(r-1)}(a)\geq \alpha_k(a)$ and so we conclude $\alpha_{k-1}(b)\geq \alpha_k(a)$, which is not the case.

Having concluded $|J_r|\geq |I|/10$, fix an integer $c$ satisfying $\alpha_{k-1}(c)=\alpha_k(a)$.
Note that $c \leq a$ because $\alpha_k(c) \leq \alpha_{k-1}(c) = \alpha_k(a)$. Also, $c \geq b = \alpha_{k-1}^{(r-1)}(a)$ because $\alpha_{k-1}(c) = \alpha_k(a) \geq \alpha_{k-1}(b)$.
These two inequalities imply that $J_r'=[v_{t/c^3}, x]$ is a subinterval of $I$ which contains $J_r$, i.e., $J_r\subseteq J_r'\subseteq I$.

We apply the induction hypothesis on $J_r'$, without any cleaning. We choose parameters $k-1$ and $c$, and the same bad set $B_0'=B_0$. The rainbow collection $\cP'$ used to apply the induction hypothesis will be $\cP'=\cP\cup \{u_iv_i\mid t/a^3<i\leq t/c^3\}$.

We have set up the interval $J_r'$ in such a way that $v_{t/\alpha_{k-1}(c)^3}$ comes after $x$, since $\alpha_{k-1}(c)=\alpha_k(a)$ and we assumed that $v_{t/\alpha_k(a)^3}$ comes after $x$. Then we have 
\begin{align*}
B'&=B_0'\cup C(\cP')\cup C(\{u_iv_i\mid i\in J_r'\})\\
&=B_0\cup C(\cP)\cup C(\{u_iv_i\mid t/a^3<i\leq t/c^3\})\cup C(\{u_iv_i\mid i\in J_r'\})\\
&=B.
\end{align*}
Finally, we add some of the $v_i$s with $t/a^3 \leq i\leq t/c^3$ to the set of endpoints $w_i$ so that we would have $t/(2c^3)$ such endpoints. Note that
\begin{multline*}
|J_r'|\geq |J_r|\geq \frac{|I|}{10}\geq \frac{C_k}{10}|B\cap [w_1, x]| + \frac{C_k}{10} \alpha_k(a)\Delta_R(I\setminus B)+\frac{C_k}{10} \frac{st}{\alpha_k(a)}\\
\geq C_{k-1}|B'\cap [w_1, x]| + C_{k-1} \alpha_{k-1}(c)\Delta_R(J'_r\setminus B)+C_{k-1} \frac{st}{\alpha_{k-1}(c)}.
\end{multline*}
Hence, the induction hypothesis can be applied to $J'_r$, to produce a rainbow collection $\cP''$ extending $\cP'$ with the property that $\ell(\cP'')\geq \ell(\cP')+\Omega_{k-1}(\frac{|J_r'|}{s})\geq \ell(\cP)+\Omega_k(\frac{|I|}{s})$, thus completing the Case 1.

\paragraph{Case 2:} Now we focus on the second case, where we have $\max_{1\leq \ell\leq r}|J_\ell|\leq |I|/10$. In particular, we have $|J_1|, |J_{r-1}|, |J_r|\leq |I|/10$, and this is the only thing we will use. Consequently, we can deduce $\sum_{\ell= 2}^{r-2} |J_\ell|\geq 7/10 |I|$, an so we have either $\sum_{\ell\geq 3, \text{ odd}}^{r-2}|J_\ell|\geq 7/20|I|$ or $\sum_{\ell\geq 2, \text{ even}}^{r-2}|J_\ell|\geq 7/20|I|$. The two situations are analogous so we simply focus on the second one.

Let us think of intervals in pairs $(J_{2\ell-1}, J_{2\ell})$ for $\ell=1, 2, \dots, \lfloor r/2\rfloor-1$. We will go through the pairs $(J_{2\ell-1}, J_{2\ell})$ one by one in reverse, i.e.\ iterate the following procedure for $\ell=\lfloor r/2\rfloor-1, \dots, 1$. The big picture of the argument is that we will apply the induction hypothesis with appropriately defined bad sets to every interval $J_{2\ell}$ which is long enough. For each interval $J_{2\ell}$, this will then give us a rainbow collection of paths $\cQ_{2\ell}$, completely contained within $J_{2\ell-1}\cup J_{2\ell}$ (except maybe for the starts of the paths, i.e.\ the $u_i$s). Once we construct the rainbow collections $\cQ_{2\ell}$, we will show that putting together all rainbow collections constructed from the induction hypothesis still yields a rainbow collection. Furthermore, we will show that there are enough long intervals $J_{2\ell}$ so that the total length of the rainbow collection $\cP\cup\bigcup_\ell \cQ_{2\ell}$ is at least ${\ell(\cP)+}\Omega_k(|I|/s)$, which is what we need to show.

Let us give the details now. We promise that $C(\cQ_{2\ell})$ will not contain any of the vertices $u_i, v_i$ for $i\leq t/\alpha^{(2\ell-2)}_{k-1}(a)^3$.

Before we can apply the induction hypothesis and get the rainbow collection, we perform a cleaning procedure. Let $b=\alpha_{k-1}^{(2\ell-1)}(a)$ and let $H^{(2\ell)}$ be the set of vertices of $J_{2\ell}\setminus B$ which have rogue degree at least $\beta\frac{st}{\alpha_{k-1}(b)^3}$ in $J_{2\ell}\setminus B$, where $\beta=20C_{k-1}$. Observe that $J_{2\ell}\setminus B$ contains at most $\frac{st}{\alpha_{k-1}(b)^3}|J_{2\ell}|$ rogue edges by \cref{lemma:rogue edges special sequence} and therefore $|H^{(2\ell)}|\leq \frac{2}{\beta}|J_{2\ell}|$.

The bad set used to apply the induction hypothesis will be
\[B^{(2\ell)}_0 = \left(B_0 \cup (B\cap J_{2\ell}) \cup C(\{u_iv_i\mid v_i\in J_{2\ell}\})\cup \bigcup_{\ell'>\ell} C(\cQ_{2\ell'})\cup H^{(2\ell)}\right)\setminus\{u_2, \dots, u_{t/b^3}\}.\] 
Let $v_1^{(2\ell)}, \dots, v_{f'}^{(2\ell)}$ be the special sequence of vertices with respect to this bad set $B^{(2\ell)}_0$ and the initial interval ending in $v_{t/\alpha_{k-1}(b)^3}$, i.e.\ up to the end of $J_{2\ell}$. Let $u_2^{(2\ell)},\dots,u_{f'}^{(2\ell)}$ be the corresponding vertices. 

\begin{claim}\label{claim:new special sequence equals old up to J_{2ell}}
For all $i\leq t/b^3$ we have $v_i=v_i^{(2\ell)}$ and $u_i = u_i^{(2\ell)}$. Furthermore, for $t/b^3<i\leq \min\{t/\alpha_{k-1}(b)^3, f'\}$, the vertex $v_i^{(2\ell)}$ does not come before $v_i$. 
\end{claim}
\begin{proof}
Our goal is to apply \cref{lemma:unchanged special sequence} with $p=t/b^3$ and therefore we verify its assumptions. By definition we have $B_0\subseteq B_0^{(2\ell)}$. Also, $B_0^{(2\ell)}$ does not contain any of the vertices $u_2, \dots, u_{t/b^3}$ by definition. For the vertices $v_1, \dots, v_{t/b^3}$, they obviously do not belong to $B\cap J_{2\ell}$ and $H^{(2\ell)}$ since they come before $J_{2\ell}$. Also, they do not belong to $B_0$ or $C(\{u_iv_i|v_i\in J_{2\ell}\})$ since $v_1, \dots, v_{t/\alpha_{k-1}(b)^3}$ is a special sequence with respect to the bad set $B_0$. Finally, they do not belong to $C(\cQ_{2\ell'})$ for any $\ell'\geq \ell$ due to the promise about $\cQ_{2\ell}$. 

The third condition of \cref{lemma:unchanged special sequence} is easy to verify since $C(\{u_iv_i\mid v_i\in J_{2\ell}\})\subseteq B_0^{(2\ell)}$ by definition. Hence, one can use \cref{lemma:unchanged special sequence} to conclude the proof.
\end{proof}

The rainbow collection used to apply the induction hypothesis is $\cP^{(2\ell)}=\cP\cup \{u_iv_i\mid t/a^3\leq i\leq t/b^3\}$. Let us check now this collection is actually rainbow. All of its starting edges come form the same special sequence and so they have different colors. Also, no two vertices of $\cP^{(2\ell)}$ conflict since the vertices of $\cP$ (except for the starting edges) are assumed to come form $[v_1, v_{t/a^3}]\setminus C(\{u_i, v_i\mid 1\leq i\leq f\})$, meaning that they have no conflicts with $\{u_iv_i\mid t/a^3\leq i\leq t/b^3\}$. Finally, we need to check that the vertices of $\cP^{(2\ell)}$, different from $u_i, v_i$, come from $[v_1, v_{t/b^3}]\setminus C(\{u_i^{(2\ell)}, v_i^{(2\ell)}\mid 1\leq i\leq f'\})$. But actually, the only such vertices in $\cP^{(2\ell)}$ belong to $\cP$. These vertices of $\cP$ definitely belong to $[v_1, v_{t/a^3}]\setminus C(\{u_i, v_i\mid 1\leq i\leq f\})$, by assumption, so we need to verify that they do not belong to $C(\{u_i^{(2\ell)}, v_i^{(2\ell)}\mid t/b^3\leq i\leq f'\})$. Since the special sequence $\{v^{(2\ell)}_i\}$ was defined with respect to the bad set $B_0^{(2\ell)}$ which contains $C(\cP)$, this follows and we can indeed apply the induction hypothesis to the rainbow collection $\cP^{(2\ell)}$.

Furthermore, we have $B^{(2\ell)}=B^{(2\ell)}_0\cup C(\cP^{(2\ell)})\cup C(\{u_i^{(2\ell)}v_i^{(2\ell)}\mid v_i^{(2\ell)}\in J_{2\ell}\})$.
At this point, we check if 
\begin{equation}\label{eq:second case induction condition}
|J_{2\ell}|\geq C_{k-1}|B^{(2\ell)}\cap (J_{2\ell}\cup J_{2\ell-1})| +C_{k-1} \alpha_{k-1}(b)\Delta_R(J_{2\ell} \setminus B^{(2\ell)})+\frac{|I|}{10\cdot 2^{r-2\ell}}+\frac{|I|}{10\alpha_k(a)}.
\end{equation}
If this condition does not hold, we fail to apply the induction hypothesis to the interval $J_{2\ell}$ because it is too short and declare the rainbow collection $\cQ_{2\ell}$ to be an empty collection. If the condition (\ref{eq:second case induction condition}) holds, let us argue that we can apply the induction hypothesis to $J_{2\ell}$ with parameters $k-1$ and $b$, together with the collection of paths $\cP^{(2\ell)}$.

Observe that the vertex $v_{t/\alpha_{k-1}(b)^3}^{(2\ell)}$ (if it is defined) comes after the vertex $v_{t/\alpha_{k-1}(b)^3}$, which is the end of the interval $J_{2\ell}$, by Claim~\ref{claim:new special sequence equals old up to J_{2ell}}\footnote{Actually, it is easy to see that $v_{t/\alpha_{k-1}(b)^3}^{(2\ell)}$ is not defined, since the special sequence $\{v_i^{(2\ell)}\}$ is defined only up to the end of $J_{2\ell}$, but this is fine from the perspective of applying the induction hypothesis.}.
We declare the vertices $v_i^{(2\ell)}$ in $J_{2\ell-1}$ to be the vertices $w_i^{(2\ell)}$, with the observation that there are indeed at least $\frac{t}{2b^3}$ of them (if there are more than $t/2b^3$ such vertices, we can take an arbitrary subset of $t/2b^3$ among them), where $b=\alpha_{k-1}^{(2\ell-1)}(a)$. 
Indeed, the number of vertices $v_i$ in intervals $J_s$ with $s < 2\ell-1$ is $\frac{t}{d^3}$, where $d = \alpha_{k-1}^{(2\ell-2)}(a) \geq 2b$, so the number of vertices $v_i \in J_{2\ell-1}$ is $\frac{t}{b^3} - \frac{t}{d^3} \geq \frac{t}{2b^3}$.

Furthermore, every vertex in $I \setminus B$ has rogue degree at most $\frac{|I|}{C_k \alpha_k(a)}$ to $I \setminus B$, by \eqref{eq:main proposition assumption}. Since $|J_{2\ell}|\geq \frac{|I|}{10 \alpha_k(a)}$ (which comes from (\ref{eq:second case induction condition})), it follows that all vertices in $[w_1^{(2\ell)},x^{(2\ell)}] \subseteq I$ (where now $w_1^{(2\ell)} \in J_{2\ell-1}$ and $x^{(2\ell)} = v_{t/\alpha_{k-1}(b)^3}$) have their rogue degree bounded by $|J_{2\ell}|/10^{k-1}$. 

The final condition we need to verify before applying induction is that 
\[|J_{2\ell}|\geq C_{k-1}\Big|B^{(2\ell)}\cap [w_1^{(2\ell)},x^{(2\ell)}]\Big| +C_{k-1} \alpha_{k-1}(b)\Delta_R(J_{2\ell} \setminus B^{(2\ell)})+C_{k-1}\frac{st}{ \alpha_{k-1}(b)}.\]
But we claim this actually follows from \eqref{eq:second case induction condition}. The middle terms of the right-hand side are the same, while for the first terms we have the inequality $|B^{(2\ell)}\cap [w_1^{(2\ell)},x^{(2\ell)}]|\leq |B^{(2\ell)}\cap (J_{2\ell-1}\cup J_{2\ell})|$ (since $[w_1^{(2\ell)},x^{(2\ell)}]\subseteq J_{2\ell-1}\cup J_{2\ell}$). 
Hence, the last thing to verify is $\frac{|I|}{10\cdot 2^{r-2\ell}}\geq C_{k-1}\frac{st}{\alpha_{k-1}(b)}$.
Recalling that $|I|\geq C_k st/\alpha_k(a)$ reduces this inequality to showing $C_k\frac{st}{10\cdot 2^{r-2\ell}\alpha_k(a)}\geq C_{k-1}\frac{st}{\alpha_{k-1}^{(2\ell)}(a)}$, which in turn follows from $\alpha_{k-1}^{(2\ell)}(a)\geq 2^{r-2\ell-2} \alpha_k(a)$. But this last inequality is easy to show, since $\alpha_{k-1}^{(2\ell)}(a)\geq 2^{r-1-2\ell} \alpha_{k-1}^{(r-1)}(a)\geq 2^{r-2-2\ell} \alpha_{k}(a)$. Hence, all conditions are satisfied and we can apply the induction hypothesis to $J_{2\ell}$. 

Applying the induction thus produces an extension $\cP^{(2\ell)}_{\rm ext}$ of the rainbow collection $\cP^{(2\ell)}$. More precisely, the only paths which get extended are the ones ending in $w_i^{(2\ell)}$, i.e.\ the paths $\{u_iv_i\mid v_i\in J_{2\ell-1}\}$. Therefore, let us denote the set of these extended paths which start in $J_{2\ell-1}$ by $\cQ_{2\ell}$, where the induction hypothesis gives us $\ell(\cQ_{2\ell})\geq \Omega_k(|J_{2\ell}|/s)$. 

Now, we verify the promises given about the collection of paths $\cQ_{2\ell}$ at the beginning of the proof.

\begin{claim}
The set $C(\cQ_{2\ell})$ does not contain any of the vertices $u_i, v_i$ for $i\leq t/\alpha_{k-1}^{(2\ell-2)}(a)^3$. 
\end{claim}
\begin{proof}
Note that $\cQ_{2\ell}$ is just a part of a larger rainbow collection $\cP^{(2\ell)}_{\rm ext}$, which contains the edges $u_iv_i$ for all $1\leq i\leq t/\alpha_{k-1}^{(2\ell-2)}(a)^3$. Since $\cP^{(2\ell)}_{\rm ext}$ is a rainbow collection, it cannot have any conflicting pairs of vertices, and therefore $u_i, v_i\notin C(\cQ_{2\ell})$ for $i\leq t/\alpha_{k-1}^{(2\ell-2)}(a)^3$.
\end{proof}

Now comes time to define the extended collection of paths $\cP'=\cP\cup \bigcup_{\ell}\cQ_{2\ell}$. We claim that this collection of paths satisfies all the conclusions that we want. In particular, we need to show that it is a rainbow collection of paths, of total length $\ell(\cP')\geq \ell(\cP)+\Omega_k(|I|/s)$, which does not contain any vertices of $B$, uses only the vertices coming before $x$, and extends only the paths ending at $w_i$. 

The first two of these statements are nontrivial and we prove them through the following two claims. To show that $\cP'$ does not contains vertices of $B$ is easy, since the collections $\cQ_{2\ell}$ are contained in $(J_{2\ell-1}\cup J_{2\ell})\setminus B^{(2\ell)}\subseteq (J_{2\ell-1}\cup J_{2\ell})\setminus B$. Also, since each of the intervals $J_{2\ell}$ comprises only vertices before $x$, so do the collections $\cQ_{2\ell}$. Finally, no paths of $\cP$ get extended and so the final condition holds too. Hence, to complete the proof we only need to show that $\cP'$ is rainbow and has sufficient length, which is the content of the following two claims.

\begin{claim}\label{claim:final sequence is rainbow}
The collection of paths $\cP'=\cP\cup \bigcup_{\ell}\cQ_{2\ell}$ is a rainbow collection.
\end{claim}
\begin{proof}
First, we should check that no two vertices of the paths conflict. By the induction hypothesis, we know that $\cQ_{2\ell}$ shares no vertices with $B^{(2\ell)}$, and $B^{(2\ell)}$ contains both $C(\cP)$ and $C(\cQ_{2\ell'})$ for all $\ell'>\ell$. Hence, $\cQ_{2\ell}$ has no vertices which could conflict with $\cP$ or $\cQ_{2\ell'}$. 

Also, observe that all paths of $\cQ_{2\ell}$ are directed forward except for the first edge which has a rogue color, so still all but at most one path of $\cP'$ are directed forward and start at a rogue edge.

Finally, we need to check that all the starting edges of paths in $\cP'$ have distinct rogue colors. Since $\cQ_{2\ell}$ is a subcollection of a rainbow collection containing $\cP$, there can be no conflict between paths in $\cP$ and $\cQ_{2\ell}$ for any $\ell$.

So the only remaining case is if there exist two paths $P_1\in\cQ_{2\ell}$, $P_2\in \cQ_{2\ell'}$ for some $\ell<\ell'$ starting with the same rogue color. Assume $P_1$ starts at $uv$, $P_2$ starts at $u'v'$ and these two edges have the same rogue color $r$. 

We begin by observing that the edge $uv$ does not appear in the original special sequence $\{u_iv_i\}$. The reason for this is that the rainbow collection $\cP^{(2\ell')}_{\rm ext}$ contains both the edges $u_iv_i$ for $v_i\in J_{2\ell-1}\cup J_{2\ell}$ and the edge $u'v'$, meaning that $u'v'$ has different color than all edges $\{u_iv_i\}$ for which $v_i$ comes before $J_{2\ell'-1}$. So we have an even stronger conclusion: none of the edges $u_iv_i$ for which $v_i$ comes before the end of $J_{2\ell}$ have the color $r=c(uv)=c(u'v')$.

Let $X=C(\{u_iv_i\mid v_i\text{ comes before the end of }J_{2\ell}\})$.
By \cref{lemit:iii}, no vertex of the set $[v_1, v_{t/\alpha_{k-1}^{(2\ell)}(a)^3}]\setminus (B_0\cup X)$ is incident to an edge of color $r$ in the graph induced on $U\setminus (B_0\cup X)$. Hence, we must have $u\in B_0\cup X$ or $v\in B_0\cup X$. 

We now argue that it is not possible for vertices of $B_0\cup X$ to be a part of $\cP^{(2\ell)}_{\rm ext}$, thus deriving a contradiction. 
Since $\cP^{(2\ell)}_{\rm ext}$ contains no vertex in $B_0^{(2\ell)}\supseteq B_0 \cup C(\{u_iv_i\mid v_i\in J_{2\ell}\})$, $u$ and $v$ cannot be in $B_0 \cup C(\{u_iv_i\mid v_i\in J_{2\ell}\})$. Also, the edges $u_iv_i$ for $i\leq t/\alpha_{k-1}^{(2\ell-1)}(b)^3$ are in the collection $\cP^{(2\ell)}_{\rm ext}$, and so $u, v\notin C(\{u_iv_i\mid i\leq t/\alpha_{k-1}^{(2\ell-1)}(b)^3\})$. But then 
\[B_0 \cup C(\{u_iv_i\mid v_i\in J_{2\ell}\})\cup C(\{u_iv_i\mid i\leq t/\alpha_{k-1}^{(2\ell-1)}(b)^3\})=B_0\cup X,\]
so $u, v\notin B_0\cup X$. This proves the claim.
\end{proof}

\begin{claim}
We have $\ell(\cP')\geq \ell(\cP)+\Omega_k\Big(\frac{|I|}{s}\Big)$.
\end{claim}
\begin{proof}
Let us denote by $L$ the set of indices $2\ell$ for which the induction hypothesis was successfully applied to $J_{2\ell}$ and for which, consequently, the collection $\cQ_{2\ell}$ is nonempty.  Note that if $\sum_{2\ell\in L}|J_{2\ell}|\geq |I|/20$ then we are done since $\ell(\cP')\geq \ell(\cP)+\sum_{2\ell\in L}\ell(\cQ_{2\ell})\geq \ell(\cP)+\Omega_k(\sum_{2\ell\in L}|J_{2\ell}|/s)$ by the induction hypothesis. 

The harder case is when $\sum_{2\ell\in L}|J_{2\ell}|\leq |I|/20$. The big picture argument for why the collections $\cQ_{2\ell}$ are still long is the following. The only reason that we could not apply the induction hypothesis to so many of the intervals $J_{2\ell}$ is that too many vertices were declared bad, i.e.\ the sets $B^{(2\ell)}$ were too large. But the main contribution to the size of $B^{(2\ell)}$ comes from $C(\cQ_{2\ell'})$, for $\ell'>\ell$. Hence, it must be that the collections $\cQ_{2\ell'}$ were long enough already.

Let us now turn this vague intuition into a precise computation. Since the total length of the intervals $J_{2\ell}$ with $2\ell\in L$ is at most $|I|/20$, and the total length of all even-indexed interval $J_{2\ell}$ is at least $7|I|/20$ (by the assumptions of Case 2), we conclude $\sum_{ 2\ell\notin L} |J_{2\ell}|\geq \frac{7}{20}|I| -  \frac{1}{20} |I| = \frac{3}{10}|I|$. Thus,

\begin{multline}\label{eqn:1}
\sum_{\ell=1}^{\lfloor r/2\rfloor -1}  \left(C_{k-1}|B^{(2\ell)}\cap (J_{2\ell}\cup J_{2\ell-1})| +C_{k-1} \alpha_{k-1}^{(2\ell)}(a)\Delta_R(J_{2\ell} \setminus B^{(2\ell)})+\frac{|I|}{10\cdot 2^{r-2\ell}}+\frac{|I|}{10\alpha_k(a)}\right)\\
\geq \sum_{ 2\ell\notin L} |J_{2\ell}|\geq \frac{3}{10}|I|. 
\end{multline}
As suggested in our preliminary intuition, the main contribution to the above sum comes from the terms $|B^{(2\ell)}\cap (J_{2\ell}\cup J_{2\ell-1})|$, since the sum of remaining two terms can be easily bounded as follows. Indeed, by summing the geometric series, the contribution of the final two terms can be bounded as
\[
\sum_{\ell= 1}^{\lfloor r/2\rfloor-1} \frac{|I|}{ 10\cdot 2^{r-2\ell}}\leq \frac{|I|}{20}, \hspace{1cm}\sum_{\ell= 1}^{\lfloor r/2\rfloor-1} \frac{|I|}{ 10\alpha_k(a)}\leq \frac{|I|}{10}.
\]
For the middle terms, we have that
\begin{align*}
\sum_{\ell= 1}^{\lfloor r/2\rfloor-1} C_{k-1} \alpha_{k-1}^{(2\ell)}(a) \Delta_R(J_{2\ell} \setminus B^{(2\ell)}) \leq \sum_{\ell= 1 }^{\lfloor r/2\rfloor-1} C_{k-1} \alpha_{k-1}^{(2\ell)}(a) \frac{\beta st}{\alpha_{k-1}^{(2\ell)}(a)^3} \leq 2C_{k-1}\beta \frac{st}{\alpha_{k}(a)^2}\leq \frac{|I|}{1000},
\end{align*}
where in the first inequality we use that all vertices from $J_{2\ell}$ of rogue degree larger than $\frac{\beta st}{\alpha_{k-1}^{(2\ell)}(a)^3}$ are contained in $H^{(2\ell)}\subseteq B^{(2\ell)}$, in the second inequality we  use that the sequence ${1}/{\alpha_{k-1}^{(2\ell)}(a)^2}$ is termwise smaller than the geometric sequence $(1/2)^{\ell}$ and so its sum can be bounded by twice the largest term, which is at most $\frac{1}{\alpha_k(a)^2}$, and in the third inequality we use \eqref{eq:main proposition assumption} and the fact that $a$ is sufficiently large with respect to $k$.
Combining these bounds with \eqref{eqn:1}, we find that 
\begin{align}\label{eqn:2}
C_{k-1}\sum_{\ell=1}^{\lfloor r/2\rfloor -1} |B^{(2\ell)}\cap (J_{2\ell}\cup J_{2\ell-1})|\geq \frac{149}{1000}|I|.
\end{align}
Now our goal is to show that the main contribution to the sum of $|B^{(2\ell)}\cap (J_{2\ell}\cup J_{2\ell-1})|$ actually comes from the sets $\bigcup_{\ell'>\ell} C(\cQ_{2\ell'})$, which are contained in $B_0^{(2\ell)}$. Therefore, we need to recall the definition of $B^{(2\ell)}$, namely $$B^{(2\ell)}=B_0^{(2\ell)}\cup C(\cP) \cup C(\{u_iv_i\mid t/a^3\leq i\leq t/\alpha_{k-1}^{(2\ell-1)}(a)^3\}) \cup C(\{u_i^{(2\ell)}v_i^{(2\ell)}\mid v_i^{(2\ell)}\in J_{2\ell}\}).$$
Moreover, we have 
\[B^{(2\ell)}_0 \subseteq B_0 \cup (B\cap J_{2\ell}) \cup C(\{u_iv_i\mid v_i\in J_{2\ell}\})\cup \bigcup_{\ell'>\ell} C(\cQ_{2\ell'})\cup H^{(2\ell)},\]
and so 
\[B^{(2\ell)}\subseteq B\cup 
C(\{u_i^{(2\ell)}v_i^{(2\ell)}\mid v_i^{(2\ell)}\in J_{2\ell}\}) \cup \bigcup_{\ell'>\ell} C(\cQ_{2\ell'})\cup H^{(2\ell)},\]
where we use the fact $B$ contains the sets $B_0, C(\{u_iv_i\mid v_i\in I\})$ and $C(\cP)$.

Using this upper bound on $B^{(2\ell)}$, we can bound the sum from (\ref{eqn:2}) as follows.
\begin{equation}
\begin{aligned}
\frac{149}{1000}\frac{|I|}{C_{k-1}}\leq \sum_{\ell=1}^{\lfloor r/2\rfloor -1} |B^{(2\ell)}\cap (J_{2\ell}\cup J_{2\ell-1})| 
&\leq \sum_{\ell=1}^{\lfloor r/2\rfloor -1} \left|B \cap \big(J_{2\ell}\cup J_{2\ell-1}\big)\right| \\
&+ \sum_{\ell=1}^{\lfloor r/2\rfloor -1} \left|H^{(2\ell)} \cap \big(J_{2\ell}\cup J_{2\ell-1}\big)\right| \\
&+\sum_{\ell=1}^{\lfloor r/2\rfloor -1}\left|\bigcup_{\ell'>\ell} C(\cQ_{2\ell'}) \cap \big(J_{2\ell}\cup J_{2\ell-1}\big)\right|\\
&+\sum_{\ell=1}^{\lfloor r/2\rfloor -1}\left|C(\{u_i^{(2\ell)}v_i^{(2\ell)}|v_i^{(2\ell)}\in J_{2\ell}\})\cap \big(J_{2\ell}\cup J_{2\ell-1}\big)\right|
\end{aligned}
\label{eqn:3}
\end{equation}
Again, the intuition is that the main term should be coming from the sets $\bigcup_{\ell'>\ell} C(Q_{2\ell'})$. So let us bound the remaining three terms of (\ref{eqn:3}). 

For the first sum, we have $\sum_{\ell=1}^{\lfloor r/2\rfloor -1} \left|B \cap \big(J_{2\ell}\cup J_{2\ell-1}\big)\right|\leq |B\cap I|\leq |I|/C_k$, from the assumption \eqref{eq:main proposition assumption}. As for the second sum, we have a bound on the size of $H^{(2\ell)}$ of the form $|H^{(2\ell)}|\leq \frac{2}{\beta}\ab{J_{2\ell}}$ and so 
\[\sum_{\ell=1}^{\lfloor r/2\rfloor -1} \left|H^{(2\ell)} \cap \big(J_{2\ell}\cup J_{2\ell-1}\big)\right|\leq \frac{2}{\beta} \sum_{\ell=1}^{\lfloor r/2\rfloor -1} \left|J_{2\ell}\right|\leq \frac{|I|}{10 C_{k-1}}\]
The final term can be bounded by invoking \cref{claim:new special sequence equals old up to J_{2ell}}, which implies that the sequence $v_i^{(2\ell)}$ contains at most $\frac{t}{\alpha_{k-1}^{(2\ell)}(a)^3}\leq \frac{t}{\alpha_k(a)^3}$ elements in $J_{2\ell}$, and thus $|C(\{u_i^{(2\ell)}v_i^{(2\ell)}\mid v_i^{(2\ell)}\in J_{2\ell}\})|\leq 2s\cdot \frac{2t}{\alpha_k(a)^3}$. Since we have at most $\alpha_k(a)$ such summands in the last sum of (\ref{eqn:3}) we find that this sum is upper-bounded by $\frac{4st}{\alpha_k(a)^2}\leq \frac{|I|}{1000C_{k-1}}$ using (\ref{eq:main proposition assumption}), and again by ensuring that $a$ is sufficiently large with respect to $k$.

Therefore, we conclude that 
\begin{align*}
    \sum_{\ell=1}^{\lfloor r/2\rfloor -1}\left|\bigcup_{\ell'>\ell} C(\cQ_{2\ell'}) \cap \big(J_{2\ell}\cup J_{2\ell-1}\big)\right|\geq \frac{149}{1000}\frac{|I|}{C_{k-1}}-\frac{|I|}{40C_{k-1}}-\frac{|I|}{10C_{k-1}}-\frac{|I|}{1000C_k}= \frac{23}{1000} \frac{|I|}{C_{k-1}}.
\end{align*}
Using this information, we can lower-bound the total length of the collections $\cQ_{2\ell}$ as follows. 
We first observe that
\begin{equation}\label{eqn:4}
\sum_{\ell=1}^{\lfloor r/2\rfloor -1}\left|\bigcup_{\ell'>\ell} C(\cQ_{2\ell'}) \cap \big(J_{2\ell}\cup J_{2\ell-1}\big)\right|\leq \left|\bigcup_{\ell=1}^{\lfloor r/2\rfloor -1} C(\cQ_{2\ell}) \right|,
\end{equation}
since $\bigcup_{\ell'>\ell} C(\cQ_{2\ell'})\subseteq \bigcup_{\ell'=1}^{\lfloor r/2\rfloor -1} C(\cQ_{2\ell'})$ and since the intervals $J_{2\ell-1}\cup J_{2\ell}$ are disjoint from one another.
Since $C(\cQ_{2\ell})$ contains at most $2s\cdot\ell(\cQ_{2\ell})$ elements, \eqref{eqn:4} yields
\[\sum_{\ell=1}^{\lfloor r/2\rfloor -1}\left|\bigcup_{\ell'>\ell} C(\cQ_{2\ell'}) \cap \big(J_{2\ell}\cup J_{2\ell-1}\big)\right|\leq \sum_{\ell=1}^{\lfloor r/2\rfloor-1} \ab{C(\cQ_{2\ell})}\leq \sum_{\ell=1}^{\lfloor r/2\rfloor -1}2s\cdot \ell(\cQ_{2\ell}).\]
This shows that $\ell(\cP')\geq \ell(\cP)+\sum_{\ell=1}^{\lfloor r/2\rfloor -1}\ell(\cQ_{2\ell})\geq \ell(\cP)+\Omega(|I|/s)$, as claimed.
\end{proof}
The last two claims verify that $\cP'$ satisfies all the required conditions, thus completing the proof of \cref{prop:main induction}.
\end{proof}

\subsection{Endgame}

In this section, we deduce \cref{thm:main} from \cref{prop:main induction}.

\begin{proof}[Proof of Theorem~\ref{thm:main}.]
Our goal is to prove that for any $k\geq 1$, we have $f(\cS, P_t)\leq A_k st\alpha_k(t)$ for some constant $A_k$, but we will prove a seemingly weaker statement $f(\cS, P_t)\leq A_k st\alpha_k(t)^2$. Of course, since $\alpha_k(t)^2\leq O_k(\alpha_{k-1}(t))$, we can derive the statement of the theorem by applying our conclusion with the parameter $k+1$ instead of $k$.

To show that $f(\cS, P_t)\leq A_k st\alpha_k(t)^2$, one needs to argue that in every edge-coloring of $K_N$, where $N= A_k st\alpha_k(t)^2$, there exists either a monochromatic copy of the tree $\cS$ or a rainbow copy of $P_t$. So assume that there is no monochromatic copy of $\cS$ and no rainbow $P_t$ in $K_N$.

Using Lemma~\ref{lemma:substructure}, we can find a set $U_0\subseteq V(K_N)$, a set of rogue colors $R$ and a partition $U_0=U_1\cup\dots\cup U_r$, with the properties listed in Lemma~\ref{lemma:substructure}. As a preliminary cleaning step, we also define $U$ to be the set of vertices of $U_0$ which have rogue degree at most $4st$. Since there are at most $t$ rogue colors (due to Lemma~\ref{lemma:substructure}) and each rogue color spans at most $s|U_0|$ edges in $U_0$, we conclude that there are at most $st|U_0|$ rogue edges in $U_0$. Therefore, at most $|U_0|/2$ vertices can be incident to more than $4st$ rogue edges, showing that $|U|\geq |U_0|/2\geq N/20$.
We sort the vertices of $U$ according to the median ordering, and by \cref{lemma:building from directed paths}, it suffices to find a rainbow collection $\cP$
of total length $\ell(\cP)\geq t$.

We now set up an application of Proposition~\ref{prop:main induction}, from which we will ultimately derive the existence of the large rainbow collection. Let the bad set $B_0$ be empty and define the special sequence $v_1, v_2, \dots, v_f$ with respect to $B_0$ in the initial interval $U$. We know that $f\leq t$, since otherwise the rainbow collection $\cP=\{u_iv_i\}$ would already have length larger than $t$. 

Consider the intervals $I_\ell=(v_{t/\alpha_{k-1}^{(\ell-1)}(t)^3}, v_{t/\alpha_{k-1}^{(\ell)}(t)^3}]$. Out of these intervals, we will consider a subcollection $I_1, \dots, I_r$ where the value of $r$ is chosen depending on $k$ in the following way. First, let $a_0=a_0(k)$ be the minimal $a$ with which \cref{prop:main induction} applies with parameters $k-1$ and $a$. Then, we choose $r$ to be the largest integer satisfying $\alpha_{k-1}^{(r-1)}(t)\geq a_0$ and  $t/\alpha_{k-1}^{(r-1)}(t)^3\leq f$. The first condition is there to ensure that Proposition~\ref{prop:main induction} applies to all intervals $I_\ell$ with $\ell\leq r$, while the second one ensures that the interval $I_r$ is defined at all. The latter condition implies that  $t/\alpha_{k-1}^{(r)}(t)^3<t$,  i.e.\ $r\leq \alpha_k(t)$.

We claim that the vertices of $U$ not contained in the union $\bigcup_{\ell=1}^r I_\ell$ can be covered by $O_k(1)$ dyadic intervals of the form $[v_p, v_{2p}]$. Indeed, if the value of $r$ was constrained by the inequality $t/\alpha_{k-1}^{(r-1)}(t)^3$, then the whole of $U$ is covered by $\bigcup_{\ell=1}^r I_\ell$. On the other hand, if the first condition constrains $r$, then we have $\alpha_{k-1}^{(r)}(t)\leq a_0$, and thus $\alpha_{k-1}^{(r+\alpha_k(a_0))}(t)\leq \alpha_{k-1}^{(\alpha_k(a_0))}(a_0)=1$, implying that $r+\alpha_k(a_0)\geq \alpha_k(t)$. Since $\alpha_k(t)-r\leq \alpha_k(a_0)$, we have $\alpha_{k-1}^{(r)}(t)^3\leq O_k(1)$, and therefore $\frac{f}{t/\alpha_{k-1}^{(r)}(t)^3}\leq O_k(f/t)\leq O_k(1)$.

Since, by Lemma~\ref{lemma:long I_k}, each dyadic interval of the form $[v_p, v_{2p}]$ contains at most $176st$ vertices, the union $\bigcup_{\ell=1}^r I_\ell$ contains all but at most $O_k(st)$ vertices of $U$. In other words, if the constant $A_k$ is large enough, we have $\sum_{\ell=1}^{r}|I_\ell|\geq |U|-O_k(st)\geq |U|/2$.

In particular, this means that for some index $\ell$, we have $|I_\ell|\geq \frac{|U|}{2r}\geq \frac{A_k}{40} st\alpha_k(t)$, since $r\leq \alpha_k(t)$. The goal is now to apply Proposition~\ref{prop:main induction} to the interval $I_\ell$.

As stated above, we define $B_0=\varnothing$ and $\cP=\{v_1, u_2v_2, \dots, u_{t/\alpha_{k-1}^{(\ell-1)}(t)^3}v_{t/\alpha_{k-1}^{(\ell-1)}(t)^3}\}$, and we choose among them an arbitrary set of $t/2\alpha_{k-1}^{(\ell-1)}(t)^3$ endpoints to play the roles of $w_i$. We also set $B=C(\{u_iv_i\mid  v_i\text{ comes before }I_\ell\})$, as required by \cref{prop:main induction}. If $A_k$ is large enough and $a=t/\alpha_{k-1}^{(\ell-1)}(t)^3$, we have 
\begin{align*}
    |I_\ell|\geq \frac{A_k}{40} st \alpha_k(t)\geq 40^k\big|B\cap [w_1, x]\big|+40^k \alpha_k(a) \Delta_R(I_\ell\setminus B)+40^k\frac{st}{\alpha_k(a)},
\end{align*}
since $|B|\leq 4st$ and $\Delta_R(I\setminus B)\leq 4st$. We also have $\Delta_R([w_1, x])\leq 4st\leq |I_\ell|/10^k$. Hence, Proposition~\ref{prop:main induction} applies to $I_\ell$, which suffices to show that there exists a rainbow collection of paths $\cP'$ of length $\ell(\cP')\geq \Omega_k(|I_\ell|/s)\geq t$. This completes the proof.
\end{proof}

\paragraph{Acknowledgments:} We are grateful to Xiaoyu He for helpful discussions on this topic.


\end{document}